\documentclass[10pt,a4paper,reqno]{amsart}
\usepackage{amsthm}
\usepackage{amsmath}
\usepackage{amssymb}
\usepackage{mathtools}
\mathtoolsset{showonlyrefs=true}

\usepackage[font=small]{caption}
\usepackage[foot]{amsaddr}

\usepackage[shortlabels]{enumitem}
\usepackage{multirow}

\makeatletter

\theoremstyle{plain}
\newtheorem{thm}{Theorem}
  \theoremstyle{definition}
  \newtheorem{defn}[thm]{Definition}
  \theoremstyle{definition}
  
  \theoremstyle{remark}
  \newtheorem{rem}[thm]{Remark}
  \theoremstyle{plain}
  
  \theoremstyle{plain}
  \newtheorem{lem}[thm]{Lemma}
  \theoremstyle{plain}
  \newtheorem{cor}[thm]{Corollary}
 \theoremstyle{definition}
  
  \theoremstyle{remark}
  \newtheorem*{rem*}{Remark}

  \theoremstyle{definition}

\usepackage{mathrsfs}

\theoremstyle{plain}

\newcommand{\N}{\mathbb{N}}
\newcommand{\R}{{\mathbb{R}}}
\newcommand{\C}{{\mathbb{C}}}
\newcommand{\Z}{{\mathbb{Z}}}
\newcommand{\dd}{{{\rm d}}}
\newcommand{\ii}{{\rm i}}

\newcommand{\diag}{\mathop\mathrm{diag}\nolimits}
\newcommand{\spn}{\mathop\mathrm{span}\nolimits}

\newcommand{\spec}{\mathop\mathrm{spec}\nolimits}
\newcommand{\sign}{\mathop\mathrm{sign}\nolimits}
\renewcommand{\Re}{\mathop\mathrm{Re}\nolimits}
\renewcommand{\Im}{\mathop\mathrm{Im}\nolimits}
\newcommand{\supp}{\mathop\mathrm{supp}\nolimits}
\newcommand{\Tr}{\mathop\mathrm{Tr}\nolimits}

\newcommand{\dist}{\mathop\mathrm{dist}\nolimits}

\newcommand{\Arg}{\mathop\mathrm{Arg}\nolimits}

\makeatother

\begin{document}

\title[Non-self-adjoint Toeplitz matrices with purely real spectrum]
{Non-self-adjoint Toeplitz matrices whose principal submatrices have real spectrum}

\author{Boris Shapiro}
\address[Boris Shapiro]{
	Department of Mathematics, 
	Stockholm University,
	Kr\"{a}ftriket 5,
	SE - 106 91 Stockholm, Sweden}
\email{shapiro@math.su.se}

\author{Franti\v sek \v Stampach}
\address[Franti\v sek \v Stampach]{
	Department of Mathematics, 
	Stockholm University,
	Kr\"{a}ftriket 5,
	SE - 106 91 Stockholm, Sweden}
\email{stampach@math.su.se}



\date{\today}

\begin{abstract}
We introduce and investigate a class of complex semi-infinite banded Toeplitz matrices satisfying the condition that the spectra of their principal submatrices 
accumulate onto a real interval when the size of the submatrix grows to $\infty$. We prove that a banded Toeplitz matrix belongs to this class if and only if its symbol 
has real values on a Jordan curve located in $\C\setminus\{0\}$. Surprisingly, it turns out that, if such a Jordan curve is present, the spectra of all the submatrices have to be real.
The latter claim is also proven for matrices given by a more general symbol. Further, the limiting eigenvalue distribution of a real banded Toeplitz matrix is related to the solution 
of a determinate Hamburger moment problem. We use this to derive a formula for the limiting measure using a parametrization of the Jordan curve. We also describe a Jacobi operator, 
whose spectral measure coincides with the limiting measure. We show that this Jacobi operator is a compact perturbation of a tridiagonal Toeplitz matrix. Our main results are illustrated 
by several concrete examples; some of them allow an explicit analytic treatment, while some are only treated numerically.\\
\textbf{Update}: The proof of Theorem 8 contains an error. An erratum is attached in the end.
\end{abstract}

\maketitle

\section{Introduction}
 
With a Laurent series 
\begin{equation}
 a(z)=\sum_{k=-\infty}^{\infty}a_{k}z^{k} 
\label{eq:def_a}
\end{equation}
with complex coefficients $a_{k}$, one can associate a semi-infinite Toeplitz matrix $T(a)$ whose elements are given by 
$$(T(a))_{i,j}=a_{i-j},\quad \forall i,j\in\N_{0}.$$ 
In the theory of Toeplitz matrices, $a$ is referred to as the symbol of $T(a)$. According to what properties of $T(a)$ are of interest,
the class of symbols is usually further restricted. For the spectral analysis of $T(a)$, the special role is played by the Wiener algebra
which consists of symbols $a$ defined on the unit circle $\mathbb{T}$ whose Laurent series~\eqref{eq:def_a} is absolutely convergent for $z\in\mathbb{T}$.
In this case, $T(a)$ is a bounded operator on the Banach space $\ell^{p}(\N)$, for any $1\leq p \leq \infty$, and its spectral properties
are well-known, see \cite{bottcher05}.

Rather than actual spectrum of $T(a)$, this paper focuses on an asymptotic spectrum of $T(a)$, meaning the set of all limit points of eigenvalues of
matrices $T_{n}(a)$, as $n\to\infty$, where $T_{n}(a)$ denotes the $n\times n$ principal submatrix of $T(a)$. An intimately related subject concerns 
the asymptotic eigenvalue distribution of Toeplitz matrices for which the most complete results were obtained if the symbol of the matrix is a Laurent polynomial.
More precisely, one considers symbols of the form
\begin{equation}
 b(z)=\sum_{k=-r}^{s}a_{k}z^{k}, \;\mbox{ where }\; a_{-r}a_{s}\neq0 \;\mbox{ and }\; r,s\in\N,
\label{eq:def_b}
\end{equation}
for which the Toeplitz matrix $T(b)$ is banded and is not lower or upper triangular. 

This subject has an impressive history going a century back to the famous work of Szeg\H{o} on the asymptotic behavior of the determinant
of $T_{n}(b)$ for $n\to\infty$. The most essential progress was achieved by Schmidt and Spitzer \cite{schmidt_ms60} who proved that the
eigenvalue-counting measures 
\begin{equation}
 \mu_{n}:=\frac{1}{n}\sum_{k=1}^{n}\delta_{\lambda_{k,n}},
\label{eq:def_mu_n}
\end{equation}
where $\lambda_{1,n},\lambda_{2,n},\dots,\lambda_{n,n}$ are the eigenvalues of $T_{n}(b)$ (repeated according to their algebraic multiplicity) and $\delta_{a}$
the Dirac measure supported at $\{a\}$, converge weakly to a \emph{limiting measure} $\mu$, as $n\to\infty$. Recall that $\mu_{n}$ converges weakly to $\mu$, as $\to\infty$,
if and only if
\[
\lim_{n\to\infty}\int_{\C}f(\lambda)\dd \mu_{n}(\lambda)=\int_{\C}f(\lambda)\dd \mu(\lambda),
\]
for every bounded and continuous function $f$ on $\C$. The measure $\mu$ is supported on the \emph{limiting set}
\begin{equation}
 \Lambda(b):=\left\{\lambda\in\C \;\big|\; \lim_{n\to\infty}\dist(\lambda,\spec(T_{n}(b)))=0\right\}\!.
\label{eq:def_Lambda_b}
\end{equation}
The limit which appears above exists, see \cite{schmidt_ms60} or \cite[Chp.~11]{bottcher05}. Moreover, Schmidt and Spitzer derived another very useful
description of $\Lambda(b)$ which reads
\[
 \Lambda(b)=\left\{\lambda\in\C \mid |z_{r}(\lambda)|=|z_{r+1}(\lambda)|\right\}\!,
\]
where $z_{1}(\lambda),z_{2}(\lambda),\dots,z_{r+s}(\lambda)$ are (not necessarily distinct) roots of the polynomial $z\mapsto z^{r}(b(z)-\lambda)$ arranged in the 
order of increase of their absolute values, i.e.,
\[
 0<|z_{1}(\lambda)|\leq |z_{2}(\lambda)|\leq\dots\leq |z_{r+s}(\lambda)|.
\]
(If there is a chain of equalities above, then the labeling of roots within this chain is arbitrary.) 

The latter description allows to deduce the fundamental analytical and topological properties of $\Lambda(b)$. First, the set $\Lambda(b)$ equals the union of a 
finite number of pairwise disjoint open analytic arcs and a finite number of the so-called exceptional points. A point $\lambda\in\Lambda(b)$ is called \emph{exceptional}, 
if either $\lambda$ is a branch point, i.e., if the polynomial $z\mapsto z^{r}(b(z)-\lambda)$ has a multiple root, or if there is no open neighborhood $U$ of $\lambda$ such that
$\Lambda(b)\cap U$ is an analytic arc starting and terminating on the boundary $\partial U$. Second, $\Lambda(b)$ is a compact and connected set, see \cite{ullman_bams67}.
Let us also mention that $\mu$ is absolutely continuous with respect to the arclength measure in $\C$ and its density has been derived by Hirschmann \cite{hirschman_ijm67}. 

As a starting point, we focus on a characterization of Laurent polynomial symbols $b$ for which $\Lambda(b)\subset\R$. In this case, we say that $b$ {\it belongs to the class} 
$\mathscr{R}$; see Definition~\ref{def:class_R}. Besides all the symbols $b$ for which $T(b)$ is self-adjoint, the class $\mathscr{R}$ contains 
also symbols of certain non-self-adjoint matrices. Surprisingly, it turns out that, for $b\in\mathscr{R}$, the spectra of all submatrices $T_{n}(b)$ are real; see Theorem~\ref{thm:summary} below.

Describing the class $\mathscr{R}$ can be viewed as a problem of characterization of the subclass of banded Toeplitz matrices with asymptotically real spectrum. Let us point out that, in this 
context, the consideration of the set of the limit spectral points, i.e. $\Lambda(b)$, is a more relevant and interesting problem than the study of the spectrum of $T(b)$ realized as an 
operator acting on $\ell^{p}(\N)$, for some $1\leq p \leq \infty$. Indeed, note that $T(b)$ is self-adjoint if and only if $b(\mathbb{T})\subset\R$ where $\mathbb{T}$ is the unit circle. On the other hand, $b(\mathbb{T})$ 
coincides with the essential spectrum of $T(b)$, see \cite[Cor.~1.10]{bottcher05}. Thus, if $T(b)$ is non-self-adjoint, then there is no chance for $\spec(T(b))$ to be purely real.

From a broader perspective, results of the present paper contribute to the study of the spectral properties of non-self-adjoint operators; a rapidly developing area which has recently 
attracted attention of many mathematicians and physicists, see, e.g. \cite{davies07,helffer13,trefethen05}. Particularly noteworthy problem consists in finding classes of non-self-adjoint operators 
that have purely real spectrum. This is usually a mathematically challenging problem (see, for example, the proof of reality of the spectrum of the imaginary cubic oscillator 
in~\cite{shin_cmp02}) which, in addition, may be of physical relevance. At the moment, the non-self-adjoint operators whose spectrum is known to be real mainly comprise either very 
specific operators \cite{shin_cmp02,giordanelligraf_ahp15, weir_apl09, sieglstampach_inprep} or operators which are in a certain sense close to being self-adjoint 
\cite{langertretter_cjp04, calicetietal_jpa06, mityaginsiegl_lmp16}. In particular, there are almost no criteria for famous non-self-adjoint families (such as Toeplitz, Jacobi, Hankel, Schr\"{o}dinger, etc.) 
guaranteeing the reality of their spectra. From this point of view and to the best of our knowledge, the present article provides the first relevant results of such a flavor for the class of banded 
Toeplitz matrices.

Our main result is the following characterization of the class $\mathscr{R}$.
\begin{thm}\label{thm:summary}
	Let $b$ be a complex Laurent polynomial as in~\eqref{eq:def_b}. The following statements are equivalent:
	\begin{enumerate}[{\upshape i)}]
		\item $\Lambda(b)\subset\R$;
		\item the set $b^{-1}(\R)$ contains (an image of) a Jordan curve;
		\item for all $n\in\N$, $\spec(T_{n}(b))\subset\R$.
	\end{enumerate}
\end{thm}
It is a very peculiar feature of banded Toeplitz matrices that the asymptotic reality of the eigenvalues (claim (i)) forces all eigenvalues of all submatrices to be real (claim (iii)). Hence, if, for instance, the $2\times2$ matrix $T_{2}(b)$
has a non-real eigenvalue, there is no chance for the limiting set $\Lambda(b)$ to be real. The implication (iii)$\Rightarrow$(i) is clearly trivial. The implication (i)$\Rightarrow$(ii) is proven in 
Theorem~\ref{thm:main_neces}, while (ii)$\Rightarrow$(iii) is established in Theorem~\ref{thm:gamma_impl_real_spec} for even more general class of symbols. These results are worked out within the first 
subsection of Section~\ref{sec:main}.

The second subsection of Section~\ref{sec:main} primarily studies the density of the limiting measure $\mu$ for real Laurent polynomial symbols $b\in\mathscr{R}$ in more detail. 
Namely, we provide an explicit description of the density of $\mu$ in terms of the symbol $b$ and the Jordan curve present in $b^{-1}(\R)$ provided that the curve admits a polar 
parametrization. Second, we show that the limiting measure is a solution of a particular Hamburger moment problem whose moment sequence is determined by the symbol $b\in\mathscr{R}$. 
The positive definiteness of the moment Hankel matrix then provides a necessary (algebraic) condition for $b$ to belong to $\mathscr{R}$. Finally, we prove an integral formula for the determinant
of the moment Hankel matrix and discuss the possible sufficiency of the condition of positive definiteness of the Hankel moment matrices for $b$ to belong to $\mathscr{R}$.

Since, for $b\in\mathscr{R}$, the support of $\mu$ is real and bounded, there is a unique self-adjoint bounded Jacobi operator $J(b)$ whose spectral measure coincides with the limiting measure.
Some properties of $J(b)$ and the corresponding family of orthogonal polynomials are investigated in Section~\ref{sec:Jac_op}. Mainly, we prove that $J(b)$ is a compact perturbation of a Jacobi 
operator with constant diagonal and off-diagonal sequences. Also the Weyl $m$-function of $J(b)$ is expressed in terms of $b$.

Section~\ref{sec:ex_numer} provides several concrete examples and numerical computations illustrating the results of Sections~\ref{sec:main} and~\ref{sec:Jac_op}. 
For  general $3$-diagonal and 4-diagonal Toeplitz matrices $T(b)$, explicit conditions in terms of $b$ guaranteeing that $b\in\mathscr{R}$  are 
deduced. Further, for the $3$-diagonal and slightly specialized 4-diagonal Toeplitz matrices, the densities of the limiting measures as well as the associated Jacobi 
matrices are obtained fully explicitly. Moreover, we show that the corresponding orthogonal polynomials are related to certain well-known families of orthogonal polynomials, which 
involve Chebyshev polynomials of the second kind and  the associated Jacobi polynomials.   Their properties such as  the three-term recurrence, the orthogonality relation,
and an explicit representation are given. Yet for another interesting example  generalizing the previous two, we are able  to obtain  some partial results. 
These examples are presented in Subsections~\ref{subsec:ex1}, \ref{subsec:ex2}, and~\ref{subsec:ex3}. Finally, the last part of the paper contains various numerical illustrations and plots of the densities of the limiting measures and the distributions of eigenvalues in the situations
whose complexity does not allow us to treat them explicitly.

\section{Main results}\label{sec:main}

\begin{defn}\label{def:class_R}
 Laurent polynomial $b$ of the form~\eqref{eq:def_b} is said to {\it belong to the class $\mathscr{R}$}, denoted by $b\in\mathscr{R}$, if and only if $\Lambda(b)\subset\R$,
 where $\Lambda(b)$ is given by~\eqref{eq:def_Lambda_b}.
\end{defn}

\begin{rem}
 Since $\Lambda(b)$ is a compact connected set, the inclusion $\Lambda(b)\subset\R$ actually implies that $\Lambda(b)$ is a closed finite interval.
\end{rem}

\subsection{Proof of Theorem~\ref{thm:summary}}

We start with the proof of the implication (i)$\Rightarrow$(ii) from Theorem~\ref{thm:summary}. For this purpose, we take a closer look at the structure of the set $b^{-1}(\R)$ 
where the symbol is real-valued. Let us stress that $b^{-1}(\R)\subset\C\setminus\{0\}$. Recall also that a \emph{curve} is a continuous mapping from a closed interval to a topological space
and a \emph{Jordan curve} is a curve which is, in addition, simple and closed. Occasionally, we will slightly abuse the term \emph{curve} sometimes meaning the mapping and sometimes the image of 
such a mapping.

Clearly, $z\in b^{-1}(\R)$ if and only if $\Im b(z)=0$ and the latter condition can be turned into a polynomial equation $P(x,y)=0$ where $P\in\R[x,y]$, $x=\Re z$, and $y=\Im z$. 
Consequently, $b^{-1}(\R)$  is a finite union of pairwise disjoint open analytic arcs, i.e., images of analytic mappings from open intervals to $\C$, and a finite number of branching 
points which are the critical points of $b$.

It might be convenient to add the points $0$ and $\infty$ to $b^{-1}(\R)$ and introduce the set 
\[
n_{b}:=b^{-1}(\R)\cup\{0,\infty\}
\]
endowed with the topology induced from the Riemann sphere $S^{2}$. Observe that the set $n_{b}$ coincides with the so-called {\it net of the rational function} $b$, 
cf. \cite{eremenkogabrielov_am02}. Recall that $n_{b}$ is a closed subset of $S^{2}$. In addition, $n_{b}$ contains neither isolated points nor curves with end-points
(i.e., $z\in n_{b}$ such that there is no open neighborhood $U$ of $z$ such that $n_{b}\cap U$ is an analytic arc starting and terminating on $\partial U$),
as one verifies by using general principles (the Mean-Value Property for harmonic functions and the Open Mapping Theorem).

In total, $n_{b}$ is a finite union of Jordan curves in $S^{2}$ of which at most one is entirely located in $\C\setminus\{0\}$. Indeed, assuming the opposite of the latter claim, one can 
always find a region (an open connected set) in $\C\setminus\{0\}$ such that $b$ has real values on its boundary. Recall that, if $f$ is a function analytic in a region 
$\Omega\subset\C$, continuous up to the boundary $\partial\Omega$, and real on the boundary, then $f$ is a (real) constant on $\Omega$, as one deduces by using the 
Maximum Modulus Principle and the Open Mapping Theorem. Thus one concludes that under this assumption there exists a non-empty region on which $b$ is a constant function, 
a contradiction. Note that, if the Jordan curve in $b^{-1}(\R)$ is present, it contains $0$ in its interior.

Notice also that, since $b(z)\sim a_{-r}/z^{r}$, as $z\to0$, $n_{b}$ looks locally at $0$ as a star graph with $2r$ edges and the angle between two consecutive
edges is $\pi/r$. Similarly, since $b(z)\sim a_{s}z^{s}$, as $z\to\infty$, $n_{b}$ looks locally at $\infty$ as a star graph with $2s$ edges with equal angles
of magnitude $\pi/s$. As an illustration of the two situations which may occur for the net of $b$, see Figure~\ref{fig:net_illust_yes_no}.

With the above information about the structure of the net $n_{b}$, the verification of the following lemma is immediate.

\begin{lem}\label{lem:path0inf_Jordan}
 The set $b^{-1}(\R)$ contains a Jordan curve if and only if every path in $S^{2}$ connecting $0$ and $\infty$ has a non-empty intersection with $b^{-1}(\R)$.
\end{lem}

Now, we are ready to show that the existence of a Jordan curve in $b^{-1}(\R)$ is necessary for $\Lambda(b)\subset\mathbb{R}$, i.e., the implication (i)$\Rightarrow$(ii) from Theorem~\ref{thm:summary}.

\begin{thm}\label{thm:main_neces}
 If $\Lambda(b)\subset\mathbb{R}$, then $b^{-1}(\R)$ contains a Jordan curve.
\end{thm}

\begin{proof}
 Set
 \[
  \mathcal{N}:=\left\{z\in\C \mid |z_{r}(b(z))|=|z_{r+1}(b(z))|\right\}\!.
 \]
 Note that $b$ maps a neighborhood of $0$ and a neighborhood of $\infty$ onto a neighborhood of $\infty$ and that
 for $\lambda$ large in the modulus, the preimages $z_{1}(\lambda),\dots,z_{r}(\lambda)$ are close to $0$, while 
 $z_{r+1}(\lambda),\dots,z_{r+s}(\lambda)$ are close to $\infty$. Hence there exist neighborhoods of $0$ and $\infty$ with
 empty intersections with $\mathcal{N}$.
 
 Next, let us show that any path in $S^{2}$ connecting $0$ and $\infty$ has a non-empty intersection with $\mathcal{N}$. Let $\gamma:[0,1]\to S^{2}$ denote such a path. 
 First, note that, for any $k\in\{1,\dots,r+s\}$, the function $\lambda\mapsto |z_{k}(\lambda)|$
 is continuous on $\C\setminus\{0\}$ and that the value $|\gamma(t)|$ appears at least once in the $(r+s)$-tuple $|z_{1}(b(\gamma(t)))|,\dots,|z_{r+s}(b(\gamma(t)))|$
 for all $t\in(0,1)$. Further, since $\gamma(0)=0$, the value $|\gamma(t)|$, for $t$ in a right neighborhood of $0$, appears among the first $r$ values 
 $|z_{1}(b(\gamma(t)))|,\dots,|z_{r}(b(\gamma(t)))|$ and does not appear in the remaining $s$ values $|z_{r+1}(b(\gamma(t)))|,\dots,|z_{r+s}(b(\gamma(t)))|$. 
 Similarly, since $\gamma(1)=\infty$, the value $|\gamma(t)|$, for $t$ in a left neighborhood of $1$, 
 appears in $|z_{r+1}(b(\gamma(t)))|,\dots,|z_{r+s}(b(\gamma(t)))|$ and is not among $|z_{1}(b(\gamma(t)))|,\dots,|z_{r}(b(\gamma(t)))|$. Since  $\gamma$ is a continuous map there must exist $t_{0}\in(0,1)$ 
 such that $|\gamma(t_{0})|=|z_{r}(b(\gamma(t_{0})))|=|z_{r+1}(b(\gamma(t_{0})))|$. Hence $\gamma((0,1))\cap\mathcal{N}\neq\emptyset$.
 
 Finally, recalling \eqref{eq:def_Lambda_b}, one has $b(\mathcal{N})\subset\Lambda(b)\subset \R$, hence $\mathcal{N}\subset b^{-1}(\R)$. Consequently, we obtain that any path in $S^{2}$
 joining $0$ and $\infty$ has a non-empty intersection with $b^{-1}(\R)$ and Theorem~\ref{thm:main_neces} follows from Lemma~\ref{lem:path0inf_Jordan}.
\end{proof}

The second part of this subsection is devoted to the proof of the implication (ii)$\Rightarrow$(iii) of Theorem~\ref{thm:summary}. Recall that every Jordan curve in $\C$ is homeomorphic to the unit circle
$\mathbb{T}:=\{z\in\C\mid |z|=1\}$ and therefore it gives rise to a homeomorphic mapping $\gamma:\mathbb{T}\to\C$. In addition, the unit circle is naturally parametrized by the polar angle $ t\mapsto e^{\ii t}$,
where $t\in[-\pi,\pi]$. In the following, we always choose the parametrization of a Jordan curve $\gamma=\gamma(t)$, with $t\in[-\pi,\pi]$, which can be viewed as the composition of the above two mappings.

Note that the mapping $z\to1/\overline{z}$ reflects the points in $\C\setminus\{0\}$ with respect to $\mathbb{T}$. For a given Jordan curve $\gamma$ with $0$ in its interior, we define a new Jordan curve
$\gamma^{*}$ by reflecting $\gamma$ with respect to $\mathbb{T}$, i.e., $\gamma^{*}:=1/\overline{\gamma}$. In other words, if $\gamma$ is parametrized as
\[
 \gamma(t)=\rho(t)e^{\ii\phi(t)},
\]
with a positive function $\rho$ and real function $\phi$, then
\[
 \gamma^{*}(t)=\frac{1}{\rho(t)}e^{\ii\phi(t)}.
\]
We also introduce the notation
\[
 r_{\gamma}:=\min_{t\in[-\pi,\pi]}|\gamma(t)| \; \mbox{ and } \; R_{\gamma}:=\max_{t\in[-\pi,\pi]}|\gamma(t)|.
\]

In the following lemma, $\langle\cdot,\cdot\rangle$ stands for the standard inner product on $\C^{n}$; further, 
to any vector $u=(u_{0},u_{1},\dots,u_{n-1})\in\C^{n}$, we associate the polynomial $f_{u}(z):=u_{0}+u_{1}z+\dots+u_{n-1}z^{n-1}$.

\begin{lem}\label{lem:quadr_form_gamma}
  Let $\gamma$ be a Jordan curve with $0$ in its interior and $a$ be a function given by the Laurent series 
  \[
   a(z)=\sum_{k=-\infty}^{\infty}a_{k}z^{k}
  \]
  which is absolutely convergent in the annulus $r_{\gamma}\leq |z| \leq R_{\gamma}$.
  Then, for all $u,v\in\C^{n}$ and $n\in\N$, one has
 \[
  \langle u, T_{n}(a)v\rangle=\langle f_{u},af_{v}\rangle_{\gamma}
  :=\frac{1}{2\pi\ii}\int_{-\pi}^{\pi}a(\gamma(t))f_{v}(\gamma(t))\overline{f_{u}(\gamma^{*}(t))}\frac{\dot{\gamma}(t)}{\gamma(t)}\dd t.
 \]
\end{lem}

\begin{proof}
 The proof proceeds by a straightforward computation. We have
 \begin{align*}
  \langle f_{u},af_{v}\rangle_{\gamma}&=\frac{1}{2\pi\ii}\int_{-\pi}^{\pi}a(\gamma(t))f_{v}(\gamma(t))\overline{f_{u}(\gamma^{*}(t))}\frac{\dot{\gamma}(t)}{\gamma(t)}\dd t\\
  &=\frac{1}{2\pi\ii}\sum_{k=-\infty}^{\infty}\sum_{l=0}^{n-1}\sum_{m=0}^{n-1}a_{k}v_{l}\overline{u_{m}}\int_{-\pi}^{\pi}\left(\gamma(t)\right)^{k+l-m-1}\dot{\gamma}(t)\dd t\\
  &=\sum_{l=0}^{n-1}\sum_{m=0}^{n-1}a_{m-l}v_{l}\overline{u_{m}}=\langle u, T_{n}(a)v\rangle,
 \end{align*}
 where we have used that
 \[
  \frac{1}{2\pi\ii}\int_{-\pi}^{\pi}\left(\gamma(t)\right)^{j-1}\dot{\gamma}(t)\dd t=\frac{1}{2\pi\ii}\oint_{\gamma}z^{j-1}\dd z=\delta_{j,0},
 \]
 for any $j\in\Z$.
\end{proof}

\begin{cor}\label{cor:quadr_form_unit_circ}
  For any element $a$ of the Wiener algebra, $u,v\in\C^{n}$, and $n\in\N$, one has
  \[
  \langle u, T_{n}(a)v\rangle=\langle f_{u},af_{v}\rangle_{\mathbb{T}}:=
  \frac{1}{2\pi}\int_{-\pi}^{\pi}a\left(e^{\ii t}\right)f_{v}\left(e^{\ii t}\right)\overline{f_{u}\left(e^{\ii t}\right)}\dd t.
 \]
\end{cor}

\begin{proof}
 Choose $\gamma(t)=e^{\ii t}$, $t\in[-\pi,\pi]$, in Lemma~\ref{lem:quadr_form_gamma}.
\end{proof}


\begin{thm}\label{thm:gamma_impl_real_spec}
  Let $\gamma$ be a Jordan curve with $0$ in its interior and $a$ be a function given by the Laurent series 
  \[
   a(z)=\sum_{k=-\infty}^{\infty}a_{k}z^{k}
  \]
  which is absolutely convergent in the annulus $\min\{1,r_{\gamma}\}\leq |z| \leq \max\{1,R_{\gamma}\}$.
  Suppose further that $a(\gamma(t))\in\R$, for all $t\in[-\pi,\pi]$. Then
  \[
  \spec(T_{n}(a))\subset\R, \quad\forall n\in\N.
  \]
\end{thm}

\begin{rem}
In particular, if the Laurent series of the symbol $a$ from Theorem~\ref{thm:gamma_impl_real_spec} converges
absolutely for all $z\in\C\setminus\{0\}$, then $\spec(T_{n}(a))\subset\R$, for all $n\in\N$, provided
that $a^{-1}(\R)$ contains a Jordan curve (which has to contain $0$ in its interior). This applies for Laurent 
polynomial symbols \eqref{eq:def_b} especially, which yields the implication (ii)$\Rightarrow$(iii) of Theorem~\ref{thm:summary}.
\end{rem}

\begin{proof}[Proof of Theorem~\ref{thm:gamma_impl_real_spec}]
 Let $\mu\in\spec(T_{n}(a))$ and $u\in\C^{n}$ be the normalized eigenvector corresponding to $\mu$. 
 Then, by using Corollary~\ref{cor:quadr_form_unit_circ}, one obtains
 \begin{equation}
  \mu=\langle u, T_{n}(a)u \rangle=\langle f_u, af_u \rangle_{\mathbb{T}}
  =\overline{\langle f_u, \overline{a}f_u \rangle_{\mathbb{T}}}.
 \label{eq:mu_eq_complex_unit_circ_in_proof}
 \end{equation}
 On the other hand, by applying Lemma~\ref{lem:quadr_form_gamma} twice, one gets
 \begin{equation}
  \langle f_u, \overline{a}f_u \rangle_{\mathbb{T}}=\langle u, T_{n}(\overline{a})u \rangle=
  \langle f_u, \overline{a}f_u \rangle_{\gamma}=\langle f_u,af_u \rangle_{\gamma},
 \label{eq:unit_circ_eq_gamma_in_proof}
 \end{equation}
 where the last equality holds since $a(\gamma(t))\in\R$, for all $t\in[-\pi,\pi]$, by the assumption.
 Finally, by using~\eqref{eq:mu_eq_complex_unit_circ_in_proof} together with~\eqref{eq:unit_circ_eq_gamma_in_proof}
 and applying Lemma~\ref{lem:quadr_form_gamma} again, one arrives at the equality
 \[
  \mu=\overline{\langle f_u,af_u \rangle_{\gamma}}=\overline{\langle u, T_{n}(a)u \rangle}=\overline{\mu}.
 \]
 Hence $\mu\in\R$.
\end{proof}

\begin{rem}\label{rem:Hessenberg_real_coef}
 Note that the entries of the Toeplitz matrix considered in Theorem~\ref{thm:gamma_impl_real_spec} are allowed to be complex. Clearly, there exists Toeplitz matrices satisfying assumptions of 
 Theorem~\ref{thm:gamma_impl_real_spec} with non-real entries, for instance, any self-adjoint Toeplitz matrix with non-real entries whose symbol belongs to the Wiener 
 algebra. On the other hand, if a Toeplitz matrix is at the same time an upper or lower Hessenberg matrix, then its entries have to be real. More precisely, if the symbol has the form
 \[
 a(z)=\frac{1}{z}+\sum_{n=0}^{\infty}a_{n}z^{n}
 \]
 and the assumptions of Theorem~\ref{thm:gamma_impl_real_spec} are fulfilled, then $a_{n}\in\R$ for all $n\in\N_{0}$. Indeed, it is easy to verify that $D_{n}:=\det T_{n}(a)$
 satisfies the recurrence
 \[
  D_{n}=(-1)^{n-1}a_{n-1}+\sum_{k=0}^{n-2}(-1)^{k}a_{k}D_{n-k-1}, \quad \forall n\in\N.
 \]
 Since, by~Theorem~\ref{thm:gamma_impl_real_spec}, all eigenvalues of $T_{n}(a)$ are real, $D_{n}\in\R$ for any $n\in\N$. Hence one can use that $a_{0}=D_{1}\in\R$ and the above recurrence
 to prove by induction that $a_{n}\in\R$ for all $n\in\N_{0}$.
\end{rem}

\subsection{The limiting measure and the moment problem}

For the analysis of this subsection, we restrict ourself with real banded Toeplitz matrices $T(b)$. More precisely, we consider symbols
\begin{equation}
 b(z)=\sum_{k=-r}^{s}a_{k}z^{k}, \;\mbox{ where }\; a_{k}\in\R, \; a_{-r}a_{s}\neq0 \;\mbox{ and }\; r,s\in\N.
\label{eq:def_b_real}
\end{equation}

Our first goal is to provide a more detailed description of the limiting measure $\mu$ for the symbols \eqref{eq:def_b_real} in terms of the 
Jordan curve $\gamma$ present in $b^{-1}(\R)$. For the sake of simplicity, we focus on the situation when the Jordan curve admits 
the polar parametrization:
\begin{equation}
 \gamma(t)=\rho(t)e^{\ii t}, \quad t\in[-\pi,\pi],
\label{eq:gamma_polar_par}
\end{equation}
where $\rho(t)>0$ for all $t\in[-\pi,\pi]$. We should mention that so far we did not observe any example of $b\in\mathscr{R}$ where the Jordan curve in $b^{-1}(\R)$
would intersect a radial ray in more than one point; so the parametrization~\eqref{eq:gamma_polar_par} would be impossible for such a curve. In particular, all the examples presented
in Section~\ref{sec:ex_numer} where $\gamma$ is known explicitly satisfy~\eqref{eq:gamma_polar_par}.

Incidentally, the problem of description of the limiting measure for $b\in\mathscr{R}$ is closely related to the classical Hamburger moment problem. Standard references on this subject are 
\cite{akhiezer90, chihara78, shohat43,simon_am98}. A \emph{solution} of the Hamburger moment problem with a moment sequence $\{h_{m}\}_{m\in\N_{0}}\subset\R$, $h_{0}=1$, is a probability measure 
$\mu$ supported in $\R$ with all moments finite and equal to $h_{m}$, i.e.,
\begin{equation}
 \int_{\R}x^{m}\dd\mu(x)=h_{m}, \quad \forall m\in\N_{0}.
\label{eq:h_m_int_mom}
\end{equation}
By the well-known Hamburger theorem, the Hamburger moment problem with a moment sequence $\{h_{m}\}_{m\in\N_{0}}$, $h_{0}=1$, has a solution if and only if the Hankel matrix
 \begin{equation}
   H_{n}:=\begin{pmatrix}
	h_{0} & h_{1} & \dots & h_{n-1}\\
	h_{1} & h_{2} & \dots & h_{n}\\
	\vdots & \vdots & \ddots & \vdots\\
	h_{n-1} & h_{n} & \dots & h_{2n-2}
      \end{pmatrix}\!,
 \label{eq:hankel_matrix}
 \end{equation}
 is positive definite for all $n\in\N$. The solution may be unique (the determinate case) or they can be infinitely many (the indeterminate case). However, if there is a solution
 whose support is compact, then it is unique, see, for example, \cite[Prop.~1.5]{simon_am98}.
 
For a later purpose, let us also recall the formula
\begin{equation}
 \lim_{n\to\infty}\frac{1}{n}\Tr\left(T_{n}(b)\right)^{m}=\frac{1}{2\pi}\int_{-\pi}^{\pi}b^{m}\left(e^{\ii t}\right)\dd t, \quad m\in\N,
\label{eq:szego}
\end{equation}
which is usually attributed to G.~Szeg\H{o}. In fact, Szeg\H{o} proved \eqref{eq:szego} for any symbol $b$ belonging to the Wiener class, imposing however, an additional assumption 
which is equivalent to the self-adjointness of $T(b)$. Later on, M.~Kac showed that \eqref{eq:szego} remains valid for any $b$ from the Wiener class without the additional assumption, 
see \cite{kac_dmj54}.

\begin{lem}\label{lem:b_in_R_sol_Hmp}
 Let $b\in\mathscr{R}$ have the form~\eqref{eq:def_b_real}. Then the Hamburger moment problem with the moment sequence given by
 \begin{equation}
  h_{0}:=1 \; \mbox{ and } \; h_{m}:=\frac{1}{2\pi}\int_{-\pi}^{\pi}b^{m}\left(e^{\ii t}\right)\dd t, \quad  m\in\N
 \label{eq:def_h_m_b}
\end{equation}
 has a solution which is unique and coincides with the limiting measure $\mu$.
\end{lem}

\begin{proof}
 First, note that $h_{m}\in\R$, for all $m\in\N_{0}$, since the coefficients of $b$ are assumed to be real. Moreover, since $b\in\mathscr{R}$, $\supp\mu\equiv\Lambda(b)\subset\R$ and it is a compact set.
 Hence, by the above discussion, it suffices to verify that $\mu$ satisfies~\eqref{eq:h_m_int_mom}. To this end, recall that $\mu$ is the weak limit of the measures $\mu_{n}$ given by~\eqref{eq:def_mu_n}.
 Then, by using the limit formula~\eqref{eq:szego}, one gets
 \[
  \int_{\R}x^{m}\dd\mu(x)=\lim_{n\to\infty}\int_{\R}x^{m}\dd\mu_{n}(x)=\lim_{n\to\infty}\frac{1}{n}\Tr\left(T_{n}(b)\right)^{m}=h_{m},
 \]
 for all $m\in\N_{0}$.
\end{proof}

\begin{cor}\label{cor:b_in_R_impl_PD_Hankel}
 If $b\in\mathscr{R}$ is of the form~\eqref{eq:def_b_real}, then the Hankel matrix~\eqref{eq:hankel_matrix} with entries given by~\eqref{eq:def_h_m_b} is positive definite for all $n\in\N$.
\end{cor}

Note that, if the coefficients of $b\in\mathscr{R}$ are real, then $\mathbb{R}\setminus\{0\}\subset b^{-1}(\R)$ and $b^{-1}(\R)$ is symmetric with respect to the real line.
The Jordan curve $\gamma$ present in $b^{-1}(\R)$ intersects the real line at exactly two points (negative and positive) which are the critical points of $b$.

\begin{thm}\label{thm:main_suff_generalized}
 Suppose that $b\in\mathscr{R}$ is as in~\eqref{eq:def_b_real} and the Jordan curve $\gamma$ contained in $b^{-1}(\R)$ admits the polar parametrization~\eqref{eq:gamma_polar_par}.
 Further, let $\ell\in\N_{0}$ be the number of critical points of $b$ in $\gamma((0,\pi))$ and $0=:\phi_{0}<\phi_{1}<\dots<\phi_{\ell}<\phi_{\ell+1}:=\pi$ 
 are such that $b'(\gamma(\phi_{j}))=0$ for all $j\in\{0,1,\dots,\ell+1\}$. Then $b\circ\gamma$ restricted to $(\phi_{i-1},\phi_{i})$
 is strictly monotone for all $i\in\{1,2,\dots,\ell+1\}$, and the limiting measure $\mu=\mu_{1}+\mu_{2}+\dots+\mu_{\ell+1}$, where $\mu_{i}$ is an absolutely continuous 
 measure supported on $[\alpha_{i},\beta_{i}]:=b(\gamma([\phi_{i-1},\phi_{i}]))$ whose density is given by
 \begin{equation}
  \frac{\dd\mu_{i}}{\dd x}(x)=\pm\frac{1}{\pi}\frac{\dd}{\dd x}(b\circ\gamma)^{-1}(x)
 \label{eq:mu_density_generalized}
 \end{equation}
 for all $x\in(\alpha_{i},\beta_{i})$ and all $i\in\{1,2,\dots,\ell+1\}$. In~\eqref{eq:mu_density_generalized} the $+$ sign is used
 when $b\circ\gamma$ increases on $(\alpha_{i},\beta_{i})$, and the $-$ sign is used otherwise.
\end{thm}

\begin{rem}
 Notice that $\supp\mu=[\min_{1\leq i\leq\ell+1}\alpha_{i},\max_{1\leq i\leq\ell+1}\beta_{i}].$ 
 Further, observe that, in particular, Theorem~\ref{thm:main_suff_generalized} provides a description of the limiting measure for all self-adjoint real banded Toeplitz matrices
 since, in this case, the Jordan curve is just the unit circle. An illustration of this situation is given in Example 5 in Subsection~\ref{subsec:var_numer}.
\end{rem}

\begin{proof}
 First, note that, for any $i\in\{1,2\dots,\ell+1\}$, the first derivative of the smooth function $b\circ\gamma:(\phi_{i-1},\phi_{i})\to\R$ 
 does not change sign on $(\phi_{i-1},\phi_{i})$ because
 \[
  (b\circ\gamma)'(t)=b'(\gamma(t))\left(\rho'(t)+\ii\rho(t)\right)e^{\ii t}=0
 \]
 if and only if $t=\phi_{j}$ for some $j\in\{0,1,\dots,\ell+1\}$. Hence $b\circ\gamma$ restricted to $(\phi_{i-1},\phi_{i})$ is either strictly increasing or strictly decreasing
 and the same holds for the the inverse $(b\circ\gamma)^{-1}:(\alpha_{i},\beta_{i})\to(\phi_{i-1},\phi_{i})$. Consequently, by formula~\eqref{eq:mu_density_generalized},
 a positive measure $\mu_{i}$ supported on $[\alpha_{i},\beta_{i}]$ is well-defined for all $i\in\{1,2\dots,\ell+1\}$. Let us denote $\nu:=\mu_{1}+\mu_{2}+\dots+\mu_{\ell+1}$.
 
 In the second part of the proof, we verify that the measure $\nu$ which is supported on $[\min_{1\leq i\leq\ell+1}\alpha_{i},\max_{1\leq i\leq\ell+1}\beta_{i}]$ has the same moments~\eqref{eq:def_h_m_b}
 as the limiting measure $\mu$. Then $\mu=\nu$ by Lemma \ref{lem:b_in_R_sol_Hmp}.
 
 Let $m\in\N_{0}$ be fixed. By deforming the positively oriented unit circle into the Jordan curve $\gamma$, one gets the equality
 \[
  h_{m}=\frac{1}{2\pi\ii}\oint_{\mathbb{T}}b^{m}(z)\frac{\dd z}{z}=\frac{1}{2\pi\ii}\oint_{\gamma}b^{m}(z)\frac{\dd z}{z}.
 \]
 Next, by using the symmetry $\overline{b(z)}=b(\overline{z})$, for $z\in\C\setminus\{0\}$, parametrization~\eqref{eq:gamma_polar_par}, and the fact that $b\circ\gamma$ is real-valued, 
 one arrives at the expression
 \[
  h_{m}=\frac{1}{\pi}\int_{0}^{\pi}b^{m}\left(\gamma(t)\right)\dd t.
 \]
 By splitting the above integral according to the positions of the critical points of $b$ on the arc of $\gamma$ in the upper half-plane, one further obtains
 \[
  h_{m}=\frac{1}{\pi}\sum_{i=1}^{\ell+1}\int_{\phi_{i-1}}^{\phi_{i}}b^{m}\left(\gamma(t)\right)\dd t.
 \]
 Finally, since $b\circ\gamma$ restricted to $(\phi_{i-1},\phi_{i})$ is monotone, we can change the variable $x=b(\gamma(t))$ in each of the above integrals getting
 \[
  h_{m}=\sum_{i=1}^{\ell+1}\int_{\alpha_{i}}^{\beta_{i}}x^{m}\dd\mu_{i}(x)=\int_{\R}x^{m}\dd\nu(x),
 \]
 which concludes the proof.
 \end{proof}
 
 Generically, there is no critical point of $b$ located on the curve $\gamma$ in the upper half-plane, i.e., $\ell=0$ in Theorem~\ref{thm:main_suff_generalized}.
 Hence the only critical points located on the curve $\gamma$ are the two intersection points of $\gamma$ and the real line. In this case, the statement of Theorem~\ref{thm:main_suff_generalized}
 gets a simpler form. Although it is just a particular case of Theorem~\ref{thm:main_suff_generalized}, we formulate this simpler statement separately below.
 
 \begin{thm}\label{thm:main_suff}
 Suppose that $b\in\mathscr{R}$ is as in~\eqref{eq:def_b_real} and the Jordan curve $\gamma$ contained in $b^{-1}(\R)$ admits the polar parametrization~\eqref{eq:gamma_polar_par}. 
 Moreover, let $b'(\gamma(t))\neq0$ for all $t\in(0,\pi)$. Then $b\circ\gamma$ restricted to $(0,\pi)$ is either strictly increasing or decreasing; the limiting measure $\mu$
 is supported on the interval $[\alpha,\beta]:=b(\gamma([0,\pi]))$ and its density satisfies
 \begin{equation}
  \frac{\dd\mu}{\dd x}(x)=\pm\frac{1}{\pi}\frac{\dd}{\dd x}(b\circ\gamma)^{-1}(x),
 \label{eq:mu_density_simple}
 \end{equation}
 for $x\in(\alpha,\beta)$, where the $+$ sign is used when $b\circ\gamma$ increases on $(0,\pi)$, and the $-$ sign  is used otherwise.
\end{thm}

\begin{rem}
 Since $b\circ\gamma$ is a one-to-one mapping from $[0,\pi]$ onto $[\alpha,\beta]$ under the assumptions of Theorem~\ref{thm:main_suff},
 we can use it for a reparametrization of the arc of $\gamma$ in the upper half-plane. Namely, we denote by $\gamma_{b}:=\gamma\circ(b\circ\gamma)^{-1}:[\alpha,\beta]\to\{z\in\C \mid \Im z\geq0\}$
 the new parametrization of the arc of $\gamma$ to distinguish the notation. Recalling~\eqref{eq:gamma_polar_par}, one observes that
 \[
  \Arg\gamma_{b}(x)=(b\circ\gamma)^{-1}(x), \quad \forall x\in[\alpha,\beta].
 \]
 Then the limiting measure is determined by the change of the argument of the Jordan curve contained in $b^{-1}(\R)$ provided that the parametrization $\gamma_{b}$ is used. More concretely,
 if we denote the distribution function of $\mu$ by $F_{\mu}:=\mu\left([\alpha,\cdot)\right)$, then formula~\eqref{eq:mu_density_simple} can be rewritten as
 \[
  F_{\mu}(x)=\begin{cases}
                         \frac{1}{\pi}\Arg\gamma_{b}(x), & \mbox{ if } \gamma_{b} \mbox{ is positively (counterclockwise) oriented},\\
                         1-\frac{1}{\pi}\Arg\gamma_{b}(x), & \mbox{ if } \gamma_{b} \mbox{ is negatively (clockwise) oriented}.
                        \end{cases}
 \]
 Yet another equivalent formulation reads
 \begin{equation}
    F_{\mu}\left(b(\gamma(t))\right)=\begin{cases}
                         \frac{1}{\pi}t, & \mbox{ if } b(\gamma(0))<b(\gamma(\pi)),\\
                         1-\frac{1}{\pi}t, & \mbox{ if } b(\gamma(\pi))<b(\gamma(0)),
                        \end{cases}
 \label{eq:mu_dist_func_simple_non_invert}
 \end{equation}
 for $t\in[0,\pi]$.
\end{rem}

Let us once more go back to the Hamburger moment problem with the moment sequence given by~\eqref{eq:def_h_m_b}. We derive an integral formula for $\det H_{n}$.

\begin{thm}
 Let $b$ be as in~\eqref{eq:def_b} and $H_{n}$ be the Hankel matrix \eqref{eq:hankel_matrix} with the entries given by~\eqref{eq:def_h_m_b}. Then, for all $n\in\N$, one has
 \begin{equation}
  \det H_{n}=\frac{1}{(2\pi)^{n}n!}\int_{\pi}^{\pi}\int_{-\pi}^{\pi}\dots\int_{-\pi}^{\pi}
  \prod_{1\leq i<j\leq n}\left(b\left(e^{\ii t_{j}}\right)-b\left(e^{\ii t_{i}}\right)\right)^{\!2}
  \dd t_{1}\dd t_{2}\dots\dd t_{n}.
 \label{eq:det_H_n_integral_form}
 \end{equation}
 Consequently, if $b$ has the form~\eqref{eq:def_b_real} and $b\in\mathscr{R}$, then
 \begin{equation}
  \int_{\pi}^{\pi}\int_{-\pi}^{\pi}\dots\int_{-\pi}^{\pi}
  \prod_{1\leq i<j\leq n}\left(b\left(e^{\ii t_{j}}\right)-b\left(e^{\ii t_{i}}\right)\right)^{\!2}
  \dd t_{1}\dd t_{2}\dots\dd t_{n}>0, \quad \forall n\in\N.
 \label{eq:int_crit_pos}
 \end{equation}
\end{thm}

\begin{proof}
 Let $n\in\N$ be fixed. First, by using the definition of the determinant, we get 
 \[
  \det H_{n}=\frac{1}{(2\pi)^{n}}\int_{\pi}^{\pi}\int_{-\pi}^{\pi}\dots\int_{-\pi}^{\pi}\det B_{n} \,\dd t_{1}\dd t_{2}\dots\dd t_{n},
 \]
 where $B_{n}=D_{n}V_{n}$, $D_{n}$ is the diagonal matrix with the entries
 \[
  \left(D_{n}\right)_{j,j}=b^{j-1}\left(e^{\ii t_{j}}\right), \quad 1\leq j\leq n,
 \]
 and $V_{n}$ is the Vandermonde matrix with the entries
 \[
  \left(V_{n}\right)_{i,j}=b^{j-1}\left(e^{\ii t_{i}}\right), \quad 1\leq i,j\leq n.
 \]
 By applying the well-known formula for the determinant of the Vandermonde matrix, one arrives at the equation
 \begin{equation}
 \det H_{n}=\frac{1}{(2\pi)^{n}}\int_{\pi}^{\pi}\int_{-\pi}^{\pi}\dots\int_{-\pi}^{\pi}
  \left[\prod_{j=1}^{n}b^{j-1}\left(e^{\ii t_{j}}\right)\right]\!
  \left[\prod_{1\leq i<j\leq n}\left(b\left(e^{\ii t_{j}}\right)-b\left(e^{\ii t_{i}}\right)\right)\right]\!
  \dd t_{1}\dd t_{2}\dots\dd t_{n}.
 \label{eq:det_H_n_first_in_proof}
 \end{equation}

 Second, we apply a symmetrization trick to the identity~\eqref{eq:det_H_n_first_in_proof}. Let $\sigma$ be a permutation of the set $\{1,2,\dots,n\}$.
 Note that the second term in the square brackets in~\eqref{eq:det_H_n_first_in_proof}, i.e., the Vandermonde determinant, is antisymmetric as a function
 of variables $t_{1},t_{2},\dots,t_{n}$. Thus, if we change variables in~\eqref{eq:det_H_n_first_in_proof} so that $t_{j}=s_{\sigma(j)}$, we obtain
 \[
  \det H_{n}=\frac{\sign\sigma}{(2\pi)^{n}}\int_{\pi}^{\pi}\dots\int_{-\pi}^{\pi}
  \left[\prod_{j=1}^{n}b^{j-1}\left(e^{\ii s_{\sigma(j)}}\right)\right]\!
  \left[\prod_{1\leq i<j\leq n}\left(b\left(e^{\ii s_{j}}\right)-b\left(e^{\ii s_{i}}\right)\right)\right]\!
  \dd s_{1}\dots\dd s_{n}.
 \]
 Now, it suffices to divide both sides of the above equality by $n!$, sum up over all permutations~$\sigma$, and recognize one more Vandermonde determinant on the right-hand side.
 This yields~\eqref{eq:det_H_n_integral_form}. The second statement of the theorem follows immediately from Corollary~\ref{cor:b_in_R_impl_PD_Hankel} and the already
 proven formula~\eqref{eq:det_H_n_integral_form}.
\end{proof}

\begin{rem}
 Note that
 \[
  B_{n}(z_{1},\dots,z_{n}):=\prod_{1\leq i<j\leq n}\left(b\left(z_{j}\right)-b\left(z_{i}\right)\right)^{2}\in\C\!\left[z_{1},z_{1}^{-1},\dots,z_{n},z_{n}^{-1}\right]\!,
 \]
 i.e., $B_{n}(z_{1},\dots,z_{n})$ is a Laurent polynomial in the indeterminates $z_{1},\dots,z_{n}$. The condition~\eqref{eq:int_crit_pos} tells us that the constant term of 
 $B_{n}(z_{1},\dots,z_{n})$ has to be positive for all $n\in\mathbb{N}$. Consequently, the inequalities~\eqref{eq:int_crit_pos} yield a necessary condition for the symbol $b$
 of the form~\eqref{eq:def_b} to belong to the class $\mathscr{R}$. In principle, these conditions can be formulated as an infinite number of inequalities in terms of the coefficients
 of $b$ (though in a very complicated form).
\end{rem}

We finish this subsection with a discussion on a possible converse of the implication from Corollary~\ref{cor:b_in_R_impl_PD_Hankel}. It is not clear now, whether, for
the symbol $b$ of the form~\eqref{eq:def_b_real}, the positive definiteness of all the Hankel matrices $H_{n}$ is a sufficient condition for $b$ to belong to $\mathscr{R}$.
If $\det H_{n}>0$ for all $n\in\N$, then the Hamburger moment problem with the moment sequence~\eqref{eq:def_h_m_b} has a solution, say $\nu$, which is unique. Indeed, the uniqueness
follows from the fact that the moment sequence $h_{m}$ does not grow too rapidly as $m\to\infty$, see \cite[Prop.~1.5]{simon_am98}. To see this, one has to realize that
the spectral radius of $T_{n}(b)$ is majorized by the (spectral) norm of $T_{n}(b)$. This norm can be estimated from above as
\begin{equation}
 \|T_{n}(b)\|\leq\sum_{k=-r}^{s}|a_{k}|=:R.
\label{eq:def_R_norm_est}
\end{equation}
Now, taking into account~\eqref{eq:szego}, one observes that the moment sequence $h_{m}$ grows at most geometrically since
\[
 |h_{m}|\leq\limsup_{n\to\infty}\frac{1}{n}\sum_{k=1}^{n}|\lambda_{k,n}|^{m}\leq R^{m},
\]
where $\lambda_{1,n},\dots,\lambda_{n,n}$ stands for the eigenvalues of $T_{n}(b)$ counted repeatedly according to their algebraic multiplicity.

Similarly as in the proof of Lemma~\ref{lem:b_in_R_sol_Hmp}, one verifies that the $m$-th moment of the limiting measure $\mu$ equals $h_{m}$. Hence, assuming $\det H_{n}>0$ for all $n\in\N$, 
both measures $\mu$ and $\nu$ have the same moments. This implies that the Cauchy transforms $C_{\mu}$ and $C_{\nu}$ of measures $\mu$ and $\nu$, respectively, coincide on a neighborhood of 
$\infty$ since
\[
 C_{\mu}(z)=\int_{\C}\frac{\dd\mu(x)}{x-z}=-\sum_{m=0}^{\infty}\frac{1}{z^{m+1}}\int_{\C}x^{m}\dd\mu(x)=-\sum_{m=0}^{\infty}\frac{h_{m}}{z^{m+1}}=\int_{\R}\frac{\dd\nu(x)}{x-z}=C_{\nu}(z).
\]
It can be shown that the equality $C_{\mu}(z)=C_{\nu}(z)$ hold for all $z\in\C$, $|z|>R$, with $R$ as in~\eqref{eq:def_R_norm_est}. This, however, does not imply $\mu=\nu$.

The measures $\mu$ and $\nu$ would coincide, if $C_{\mu}(z)=C_{\nu}(z)$ for almost every $z\in\C$ (with respect to the Lebesgue measure in $\C$). This can be obtained by imposing
an additional assumption on $\Lambda(b)$ requiring its complement to be connected. Indeed, since the Cauchy transform of a measure is analytic outside the support of the measure,
the equality $C_{\mu}(z)=C_{\nu}(z)$ can be extended to all $z\notin\Lambda(b)\cup\supp\nu$ by analyticity provided that $\C\setminus\Lambda(b)$ is connected. Clearly, both
sets $\Lambda(b)$ and $\supp\nu$ have zero Lebesgue measure and therefore $\mu=\nu$.

It follows from the above discussion that, if there exists $b$ of the form~\eqref{eq:def_b} such that $\det H_{n}>0$, for all $n\in\N$, and $b\notin\mathscr{R}$, then the set
$\C\setminus\Lambda(b)$ has to be disconnected. Let us stress that it is by no means clear for which symbols $b$ the set $\C\setminus\Lambda(b)$ is connected. To our best knowledge,
the only exception corresponds to symbols $b$ which are trinomials for which the set $\C\setminus\Lambda(b)$ is known to be connected as pointed out in \cite{schmidt_ms60}.
On the other hand, the relatively simple examples of $b$ where $\Lambda(b)$ separates the plane $\C$ are known, see \cite[Prop.~5.2]{bottcher_laa02}. One might believe that, if $T_{n}(b)$ is a lower (or upper)
Hessenberg matrix, i.e, $r=1$ (or $s=1$) in~\eqref{eq:def_b_real}, then $\C\setminus\Lambda(b)$ is connected. It seems that it is not the case neither. It is not our intention to prove it analytically;
we provide only a numerical evidence given by Figure~\ref{fig:separ}.

\section{Associated Jacobi operator and orthogonal polynomials}\label{sec:Jac_op}

In Lemma~\ref{lem:b_in_R_sol_Hmp}, we have observed that, for $b\in\mathscr{R}$ of the form~\eqref{eq:def_b_real}, the limiting measure is a solution of the Hamburger moment problem.
It is well known that the Hamburger moment problem is closely related with Jacobi operators and orthogonal polynomials \cite{akhiezer90}. The aim of this section is to investigate
spectral properties of the Jacobi operator whose spectral measure coincides with the limiting measure $\mu$ provided that $b\in\mathscr{R}$. For the general theory of Jacobi operators,
we refer the reader to \cite{teschl00}.

Recall that $\ell^{2}(\N)$ is the Hilbert space of square summable sequences indexed by~$\N$ endowed with the standard inner product denoted here by $\langle\cdot,\cdot\rangle$,
and $\{e_{k} \mid k\in\N\}$ stands for the canonical basis of $\ell^{2}(\N)$.

\begin{thm}\label{thm:Jac_op_rel_b_in_R}
 For every $b\in\mathscr{R}$ of the form~\eqref{eq:def_b_real}, there exists a bounded self-adjoint Jacobi operator $J(b)$ acting on $\ell^{2}(\N)$ such that
 its spectral measure $E_{J(b)}$ satisfies
 \begin{equation}
  \mu=\langle e_{1}, E_{J(b)}e_{1} \rangle,
 \label{eq:mu_eq_spec_meas}
 \end{equation}
 where $\mu$ is the limiting measure of $T(b)$; in particular, $\spec(J(b))=\Lambda(b)$. Moreover, for $\lambda\notin\Lambda(b)$, 
 the Weyl $m$-function $m_{b}(\lambda):=\langle e_{1},(J(b)-\lambda)^{-1}e_{1}\rangle$ of $J(b)$ satisfies the equation
 \begin{equation}
  m_{b}(\lambda)=\frac{1}{2\pi}\int_{0}^{2\pi}\frac{\dd\theta}{b\left(\rho(\lambda)e^{\ii\theta}\right)-\lambda},
  \label{eq:weyl_m_b}
 \end{equation}
 where $\rho(\lambda)>0$ is an arbitrary number such that $|z_{r}(\lambda)|<\rho(\lambda)<|z_{r+1}(\lambda)|$.
\end{thm}

\begin{proof}
 By the general theory, for every determinate Hamburger moment problem with a moment sequence $\{h_{m}\}_{m=0}^{\infty}$, $h_{0}=1$, there exists a self-adjoint 
 Jacobi operator $J$ determined uniquely by its diagonal and off-diagonal sequence: $b_{n}=J_{n,n}$ and $a_{n}=J_{n,n+1}=J_{n+1,n}$, $n\in\N$, see \cite[Chp.~4]{akhiezer90}. 
 The sequences $\{a_{n}\}_{n\geq1}$, $\{b_{n}\}_{n\geq1}$ are expressible in terms of the moment sequence $\{h_{m}\}_{m=0}^{\infty}$ by the formulas \cite[Eq.~(2.118)]{teschl00}
 \begin{equation}
  a_{1}=\frac{\sqrt{\det H_{2}}}{\det H_{1}}=\sqrt{h_{2}-h_{1}^{2}}, \quad a_{n}=\frac{\sqrt{\det H_{n-1}\det H_{n+1}}}{\det H_{n}},\quad n\geq2,
 \label{eq:a_n_det_Hankel}
 \end{equation}
 and
 \begin{equation}
 b_{1}=\frac{\det\tilde{H}_{1}}{\det H_{1}}=h_{1}, \quad b_{n}=\frac{\det\tilde{H}_{n}}{\det H_{n}}-\frac{\det\tilde{H}_{n-1}}{\det H_{n-1}}, \quad n\geq2,
 \label{eq:b_n_det_Hankel}
 \end{equation}
 where $\tilde{H}_{n}$ is obtained from $H_{n+1}$ by deleting its $n$th row and its $(n+1)$st column. The projection-valued spectral measure $E_{J}$ of the self-adjoint operator $J$ determines
 the measure $\mu=\langle e_{1}, E_{J}e_{1} \rangle$ which coincides with the solution of the determinate Hamburger moment problem with the moment sequence $\{h_{m}\}_{m=0}^{\infty}$.
 Moreover, $\spec(J)=\supp\mu$ and therefore $\supp\mu$ is bounded if and only if $J$ is a bounded operator. These facts together with Lemma~\ref{lem:b_in_R_sol_Hmp} and the equality 
 $\supp\mu=\Lambda(b)$ yield all the claims of the statement  except for the formula~\eqref{eq:weyl_m_b}.
 
 Let us now derive the expression~\eqref{eq:weyl_m_b} for the Weyl $m$-function of $J(b)$.
 By expressing the resolvent operator of $J(b)$ in terms of its spectral measure, one gets
 \begin{equation}
  \langle e_{1}, (J(b)-\lambda)^{-1}e_{1} \rangle=\int_{\R}\frac{\dd\mu(x)}{x-\lambda}, \quad \lambda\notin\spec(J(b)).
 \label{eq:weil_m_func_calc}
 \end{equation}
 In other words, the Weyl $m$-function of $J(b)$ is nothing else but the Cauchy transform of the limiting measure $\mu$. 
 
 Let $\lambda\in\C\setminus\Lambda(b)$ be fixed. Take an arbitrary $\rho=\rho(\lambda)>0$ such that $|z_{r}(\lambda)|<\rho<|z_{r+1}(\lambda)|$.
 In the course of  the proof of \cite[Lem.~11.11]{bottcher05}, it was shown that
 \[
  \frac{1}{2\pi}\int_{0}^{2\pi}\log\left(b\left(\rho e^{\ii\theta}\right)-\lambda\right)\dd\theta=2m\pi\ii+\log a_{s}+\sum_{j=r+1}^{r+s}\log\left(-z_{j}(\lambda)\right)\!,
 \]
 where $m$ is an integer. Differentiating the above equation with respect to $\lambda$, one obtains 
 \begin{equation}
  \frac{1}{2\pi}\int_{0}^{2\pi}\frac{\dd\theta}{b\left(\rho e^{\ii\theta}\right)-\lambda}=-\sum_{j=r+1}^{r+s}\frac{z_{j}'(\lambda)}{z_{j}(\lambda)}.
 \label{eq:aux_in_proof_1}
 \end{equation}
 On the other hand, it is also well-known that
 \[
  \int_{\Lambda(b)}\frac{\dd\mu(x)}{x-\lambda}=\sum_{j=1}^{r}\frac{z_{j}'(\lambda)}{z_{j}(\lambda)}, \quad \lambda\notin\Lambda(b),
 \]
 see, e.g., \cite[Prop.~4.2]{duitskuijlaars_siamjmaa08}. By using the above formula together with the identity
 \[
  \sum_{j=1}^{r}\frac{z_{j}'(\lambda)}{z_{j}(\lambda)}=-\sum_{j=r+1}^{r+s}\frac{z_{j}'(\lambda)}{z_{j}(\lambda)},
 \]
 which, in its turn, is obtained by the logarithmic differentiation of the equation
 \[
  \prod_{j=1}^{r+s}z_{j}(\lambda)=(-1)^{r+s}\frac{a_{-r}}{a_{s}}
 \]
 expressing the relation between the roots and the leading and constant coefficients of the polynomial $z^{r}(b(z)-\lambda)$, we obtain 
  \begin{equation}
  \int_{\Lambda(b)}\frac{\dd\mu(x)}{x-\lambda}=-\sum_{j=r+1}^{r+s}\frac{z_{j}'(\lambda)}{z_{j}(\lambda)}, \quad \lambda\notin\Lambda(b).
 \label{eq:aux_in_proof_2}
 \end{equation}
 Combining \eqref{eq:aux_in_proof_1} and \eqref{eq:aux_in_proof_2}, one arrives at the equality
 \[
  \frac{1}{2\pi}\int_{0}^{2\pi}\frac{\dd\theta}{b\left(\rho e^{\ii\theta}\right)-\lambda}=\int_{\Lambda(b)}\frac{\dd\mu(x)}{x-\lambda},
 \]
 which, together with~\eqref{eq:weil_m_func_calc}, yields~\eqref{eq:weyl_m_b}.
\end{proof}

\begin{rem}
Recall that the density of $\mu$ can be recovered from $m_{b}$ by using the Stieltjes--Perron inversion formula, see, for example, \cite[Chp.~2 and Append.~B]{teschl00}. 
Namely, one has
\begin{equation}
 \frac{\dd\mu(x)}{\dd x}=\frac{1}{\pi}\lim_{\epsilon\to0+}\Im m_{b}(x+\ii\epsilon),
\label{eq:weyl-m_stieltjes_inversion}
\end{equation}
for all $x\in\Lambda(b)$ which are not exceptional points.
\end{rem}

\begin{rem}
  Let us now apply the close relation between the spectral properties of Jacobi operators and the orthogonal polynomials. Theorem~\ref{thm:Jac_op_rel_b_in_R} tells 
 us that for any $b\in\mathscr{R}$ of the form~\eqref{eq:def_b_real}, there exists a family of orthogonal polynomials $\{p_{n}\}_{n=0}^{\infty}$, determined by the three-term recurrence
 \begin{equation}
  p_{n+1}(x)=(x-b_{n+1})p_{n}(x)-a_{n}^{2}p_{n-1}(x), \quad n\in\N_{0},
 \label{eq:og_pol_recur}
 \end{equation}
 with the initial conditions $p_{-1}(x)=0$ and $p_{0}(x)=1$, where coefficients $\{a_{n}\}_{n=1}^{\infty}$ and $\{b_{n}\}_{n=1}^{\infty}$ are given by~\eqref{eq:a_n_det_Hankel} and~\eqref{eq:b_n_det_Hankel} 
 ($a_{0}$ is arbitrary). This family satisfies the orthogonality relation
 \begin{equation}
  \int_{\alpha}^{\beta}p_{n}(x)p_{m}(x)\rho(x)\dd x=\frac{\det H_{n+1}}{\det H_{n}}\delta_{m,n}, \quad m,n\in\N_{0},
 \label{eq:og_rel}
 \end{equation}
 where $\rho$ stands for the density of the limiting measure $\mu$ supported on $\Lambda(b)=[\alpha,\beta]$. For $n=0$, one has to set $\det H_{0}:=1$ in~\eqref{eq:og_rel}.  
 Particular examples of these families of polynomials are examined in Section~\ref{sec:ex_numer}.
\end{rem}

The map  sending $b\in\mathscr{R}$ to the  corresponding Jacobi parameters $\{a_{n}\}_{n=1}^{\infty}$ and $\{b_{n}\}_{n=1}^{\infty}$ is interesting; however, 
it is unlikely that in a concrete situation,  one can use  general formulas \eqref{eq:a_n_det_Hankel} and \eqref{eq:b_n_det_Hankel} 
to obtain the diagonal and off-diagonal sequences of~$J(b)$ explicitly. Consequently, it is desirable at least to know  the asymptotic behavior of $a_{n}$ and $b_{n}$,
as $n\to\infty$.

\begin{thm}\label{thm:compact}
 Let $b\in\mathscr{R}$ of the form~\eqref{eq:def_b_real} and $\Lambda(b)=[\alpha,\beta]$. Then $J(b)$ is a compact perturbation of the Jacobi matrix 
 with the constant diagonal sequence  $(\alpha+\beta)/2$ and the constant off-diagonal sequence  
 $(\beta-\alpha)/4$, i.e.,
 \[
  \lim_{n\to\infty}a_{n}=\frac{\beta-\alpha}{4} \quad \mbox{ and } \quad \lim_{n\to\infty}b_{n}=\frac{\alpha+\beta}{2}.
 \]
\end{thm}

\begin{proof}
 The statement is a consequence of Rakhmanov's theorem \cite[Thm.~4]{denisov_pams04}; see also Thm.~1 in loc. cit. referring to a weaker older result due to P.~Nevai \cite{nevai_mams79}
 which is still sufficient for our purposes. The latter theorem implies that, if 
 \begin{enumerate}[{\upshape i)}]
  \item $J$ is a bounded self-adjoint Jacobi operator such that $\spec_{\text{ess}}(J)=[a,b]$,
  \item $\rho(x)>0$ for almost every $x\in[a,b]$, where $\rho$ denotes the density of the Lebesgue absolutely continuous component of the measure $\langle e_{1}, E_{J}e_{1}\rangle,$ 
 $E_{J}$ being the spectral measure of $J$, 
 \end{enumerate}
 then $J$ is a compact perturbation of the Jacobi operator with constant diagonal sequence  $(a+b)/2$ and constant off-diagonal sequence  $(b-a)/4$.

 In  case of $J(b)$, Theorem~\ref{thm:Jac_op_rel_b_in_R} implies that $\spec(J(b))=\Lambda(b)=[\alpha,\beta]$. In fact, since $\spec(J(b))$ does not contain 
 isolated points, the discrete part of the spectrum of the self-adjoint operator $J(b)$ is empty and hence  $\spec_{\text{ess}}(J(b))=[\alpha,\beta]$.
 Taking into account~\eqref{eq:mu_eq_spec_meas}, it suffices to show that the density  of the limiting measure $\mu$ is positive almost everywhere on $(\alpha,\beta)$.
 The latter fact is true since the density is positive on every analytic arc of $\Lambda(b)$, i.e., everywhere except possibly at the exceptional points which are finitely many, 
 see \cite[Cor.~4c]{hirschman_ijm67}.
\end{proof}

\begin{rem}
 Using the terminology of orthogonal polynomials, the statement of Theorem~\ref{thm:compact} says that the family of orthogonal polynomials $\{p_{n}\}_{n=0}^{\infty}$ determined by~\eqref{eq:og_pol_recur}
 belongs to the Blumenthal--Nevai class $M((\beta-\alpha)/2,(\alpha+\beta)/2)$, see \cite{nevai_mams79}.
\end{rem}

\section{Examples and numerical computations}\label{sec:ex_numer}

\subsection{Example 1 (tridiagonal case)}\label{subsec:ex1}

First, we take a look at the simplest nontrivial situation when $T(b)$ is a tridiagonal Toeplitz matrix. Since the entry on the main diagonal only causes a shift of the spectral parameter
and the matrix $T(b)$ itself can be scaled by a nonzero constant we may assume, without loss of generality, that $T(b)$ belongs to the one parameter family 
of tridiagonal Toeplitz matrices with the symbol
\[
 b(z)=\frac{1}{z}+az,
\]
where $a\in\C\setminus\{0\}$.

First, we decide for what parameter $a$, the symbol $b\in\mathscr{R}$. By Theorem~\ref{thm:summary}, $\spec T_{2}(b)\subset\R$ is a necessary condition
from which one easily deduces that $a>0$, if $b\in\mathscr{R}$. Next, for $a>0$, the symbol $b$ is real-valued on the circle $\gamma(t)=a^{-1/2}e^{\ii t}$,
$t\in[-\pi,\pi]$. Thus, by using Theorem~\ref{thm:summary} once more, we conclude that $b\in\mathscr{R}$ if and only if $a>0$.

One way to deduce the limiting measure $\mu$, for $a>0$, is to solve the corresponding Hamburger moment problem with the moment sequence \eqref{eq:def_h_m_b}.
A straightforward use of the binomial formula provides us with the constant term of the Laurent polynomial $b^{m}(z)$, for $m\in\N$, giving the identity
\begin{equation}
 h_{m}=\begin{cases}
        \binom{2k}{k}a^{k},& \quad \mbox{ if } m=2k,\\
        0,& \quad \mbox{ if } m=2k-1.\\
       \end{cases}
\label{eq:h_m_tridiag}
\end{equation}
Recall that the Wallis integral formula reads
\[
 \int_{0}^{\pi/2}\sin^{2k}\theta\dd\theta=\binom{2k}{k}\frac{\pi}{2^{2k+1}}, \quad k\in\N_{0}.
\]
Changing the variable in the above integral by $x=2\sqrt{a}\sin\theta$, one obtains
\[
 \frac{2}{\pi}\int_{0}^{2\sqrt{a}}\frac{x^{2k}\dd x}{\sqrt{4a-x^{2}}}=\binom{2k}{k}a^{k},
\]
or, equivalently,
\[
 \int_{-2\sqrt{a}}^{2\sqrt{a}}x^{2k}\rho(x)\dd x=\binom{2k}{k}a^{k}, \quad \forall k\in\N_{0},
\]
where 
\begin{equation}
 \rho(x):=\frac{1}{\pi\sqrt{4a-x^{2}}}, \quad x\in(-2\sqrt{a},2\sqrt{a}).
\label{eq:density_tridiag}
\end{equation}
Hence, taking also into account that $\rho$ is an even function, we have
\[
 \int_{-2\sqrt{a}}^{2\sqrt{a}}x^{m}\rho(x)\dd x=h_{m}, \quad \forall m\in\N_{0},
\]
with $h_{m}$ given by \eqref{eq:h_m_tridiag}. Consequently, the measure $\mu$ supported on $[-2\sqrt{a},2\sqrt{a}]$ with the above density $\rho$
is the unique solution of the Hamburger moment problem with the moment sequence~\eqref{eq:h_m_tridiag}. By Lemma~\ref{lem:b_in_R_sol_Hmp},
this measure is the desired limiting measure $\mu$. For the distribution function of $\mu$, one obtains
\[
 \mu\left([-2\sqrt{a},x)\right)=\frac{1}{2}+\frac{1}{\pi}\arcsin\left(\frac{x}{2\sqrt{a}}\right)\!, \quad x\in[-2\sqrt{a},2\sqrt{a}].
\]
This measure is well-known as the arcsine measure supported on the interval $[-2\sqrt{a},2\sqrt{a}]$. In particular, it is the so-called equilibrium measure of $[-2\sqrt{a},2\sqrt{a}]$, see e.g. \cite[Sec.~I.1]{saff97}.

An alternative way  to deduce $\mu$ is based on formula~\eqref{eq:weyl-m_stieltjes_inversion}. To obtain a suitable expression for the Weyl $m$-function, one can use  the generating 
function for the moments~\eqref{eq:h_m_tridiag} which reads as 
\[
 \sum_{m=0}^{\infty}h_{m}z^{m}=\sum_{k=0}^{\infty}\binom{2k}{k}a^{k}z^{2k}=\frac{1}{\sqrt{1-4az^{2}}}, \quad |z|<1/(2\sqrt{a}).
\]
This identity together with the von Neumann series expansion of the resolvent operator yields
\[
 m_{b}(z)=-\frac{1}{z}\sum_{m=0}^{\infty}\frac{h_{m}}{z^{m}}=-\frac{1}{\sqrt{z^{2}-4a}},
\]
for $z\in\C\setminus[-2\sqrt{a},2\sqrt{a}]$. By substituting the latter expression into \eqref{eq:weyl-m_stieltjes_inversion} and evaluating the limit, one rediscovers \eqref{eq:density_tridiag}.

Further, let us examine the structure of the Jacobi matrix $J(b)$ and the corresponding family of orthogonal polynomials.
First, since the density~\eqref{eq:density_tridiag} is an even function on the interval symmetric with respect to $0$, the 
diagonal sequence $\{b_{n}\}_{n=1}^{\infty}$ vanishes, as it follows, for example, from \cite[Thm.~4.2\ (c)]{chihara78}.
To compute the off-diagonal sequence $\{a_{n}\}_{n=1}^{\infty}$, we need to evaluate $\det H_{n}$ where $H_{n}$ is the Hankel
matrix \eqref{eq:hankel_matrix} with the entries given by moments~\eqref{eq:h_m_tridiag}. The evaluation of $\det H_{n}$ is
treated in Lemma~\ref{lem:det_H_n_tridiag} in the Appendix. By using Lemma~\ref{lem:det_H_n_tridiag} together with~\eqref{eq:a_n_det_Hankel}, 
one immediately gets $a_{1}=2\sqrt{a}$ and $a_{n}=\sqrt{a}$ for $n>1$. Thus, the self-adjoint Jacobi operator associated with $b$ has the matrix representation
\[
 J(b)=\begin{pmatrix}
        0 & 2\sqrt{a} \\
        2\sqrt{a} & 0 & \sqrt{a} \\
           & \sqrt{a} & 0 & \sqrt{a} \\
	   &   & \sqrt{a} & 0 & \sqrt{a} \\
	   &   &   & \ddots & \ddots & \ddots 
      \end{pmatrix}\!,
\]
where $a>0$.

The corresponding family of orthogonal polynomials $\{p_{n}\}_{n=0}^{\infty}$ generated by the recurrence~\eqref{eq:og_pol_recur}
(and initial conditions given therein) satisfies the orthogonality relation
\[
 \int_{-2\sqrt{a}}^{2\sqrt{a}}p_{n}(x)p_{m}(x)\frac{\dd x}{\sqrt{4a-x^{2}}}=\pi(2-\delta_{n,0})a^{n}\delta_{n,m}, \quad \forall m,n\in\N_{0},
\]
which one verifies by using~\eqref{eq:og_rel} together with the formula~\eqref{eq:density_tridiag} and Lemma~\ref{lem:det_H_n_tridiag}.
Polynomials $\{p_{n}\}_{n=0}^{\infty}$ do not belong to any family listed in the hypergeometric Askey scheme \cite{koekoek10}.
However, they can be written as the following linear combination of Chebyshev polynomials of the second kind,
\[
 p_{n}(x)=a^{n/2}\left(U_{n}\left(\frac{x}{2\sqrt{a}}\right)-3U_{n-2}\left(\frac{x}{2\sqrt{a}}\right)\right)\!,
\]
for $n\in\N$ (here $U_{-1}(x):=0$). The above equation and the hypergeometric representation of Chebyshev polynomials, see \cite[Eq.~(9.8.36)]{koekoek10}, 
can be used to obtain the explicit formula
\[
 p_{n}(x)=\sum_{k=0}^{\lfloor\frac{n}{2}\rfloor}(-1)^{k}\frac{(n+2k)(n-1-k)!}{k!(n-2k)!}a^{k}x^{n-2k}, \quad n\in\N,
\]
where $\lfloor y \rfloor$ denotes the greatest integer less or equal to a real number $y$.

\subsection{Example 2 (4-diagonal case)}\label{subsec:ex2}

Let us examine the case of symbols \eqref{eq:def_b} with $r=1$ and $s=2$. Without loss of generality, we can set $a_{-1}=1$ and $a_{0}=0$ which yields the symbol $b$  of the form
\begin{equation}
 b(z)=\frac{1}{z}+\alpha z+\beta z^{2},
\label{eq:def_b_4diag}
\end{equation}
where $\alpha\in\C$ and $\beta\in\C\setminus\{0\}$. 

First, we discuss for which parameters $\alpha$ and $\beta$ we have $b\in\mathscr{R}$. By Remark~\ref{rem:Hessenberg_real_coef}, if $b\in\mathscr{R}$, $\alpha$ and $\beta$ have to be real.
Further, according to Theorem~\ref{thm:summary}, if $b\in\mathscr{R}$, then $\spec T_{3}(b)\subset\R$.
The characteristic polynomial of $T_{3}(b)$ reads
\[
 \det(T_{3}(b)-z)=\beta+2\alpha z-z^{3}.
\]
By inspection of the discriminant of the above cubic polynomial with real coefficients, one concludes that its roots are real if and only if $\alpha^{3}\geq 27\beta^{2}$.

Next, we show that, for $\beta\in\R\setminus\{0\}$ and $\alpha^{3}\geq 27\beta^{2}$, $b^{-1}(\R)$ contains a Jordan curve. Note that, if $\alpha^{3}\geq 27\beta^{2}$,
the equation
\[
 z^{2}b'(z)=-1+\alpha z^{2}+2\beta z^{3}=0.
\]
have all roots non-zero real and they cannot degenerate to a triple root. Consequently, all critical points of $b$ are real. At the same time, these points are the intersection points of the net~$n_{b}$.
Taking into account that only one curve  in $n_{b}$ (the real line) passes through $0$, because $b(z)\sim 1/z$, as $z\to0$, and two curves (the extended real line and one more) pass through 
$\infty$, because $b(z)\sim \beta z^{2}$, as $z\to\infty$, one concludes that there has to be another arc in $b^{-1}(\R)$ passing through a real critical point of $b$. This arc necessarily
closes into a Jordan curve located in $b^{-1}(\R)$. Altogether, Theorem~\ref{thm:summary} implies that $b\in\mathscr{R}$ if and only if $\beta\in\R\setminus\{0\}$ and $\alpha^{3}\geq 27\beta^{2}$.

Next, we will derive the limiting measure $\mu$ explicitly in a special case when the symbol takes the form
\begin{equation}
 b(z)=\frac{1}{z}(1+az)^{3},
\label{eq:def_b_4diag_spec}
\end{equation}
where $a\in\R\setminus\{0\}$. This corresponds to $\alpha=3a^{2}$, $\beta=a^{3}$ in~\eqref{eq:def_b_4diag}, but we additionally add a real constant to $b$. Clearly, if we add a real constant to $b$,
the net $n_{b}$ does not change and $\Lambda(b)$ is just shifted by the constant. Hence, we may use the previous discussion to conclude that $b$, given by~\eqref{eq:def_b_4diag_spec}, belongs to the 
class $\mathscr{R}$ for all $a\in\R\setminus\{0\}$.

Note that the critical points of $b$ are $-1/a$ and $1/(2a)$, hence the measure $\mu$ is supported on the closed interval with the endpoints $b(-1/a)=0$ and $b(1/(2a))=27a/4$.
By using the binomial formula, one verifies that the constant term of $b^{m}(z)$ equals
\[
 h_{m}=\binom{3m}{m}a^{m}, \quad m\in\N_{0}.
\]
The generating function for the moments $h_{m}$ can be expressed in terms of the Gauss hypergeometric series as
\begin{equation}
 \sum_{m=0}^{\infty}h_{m}z^{m}={}_{2}F_{1}\!\left(\frac{1}{3},\frac{2}{3};\frac{1}{2};\frac{27}{4}az\right)\!.
\label{eq:gener_h_m_2F1_4diag}
\end{equation}

For the sake of simplicity of the forthcoming formulas and without loss of generality, we set $a=4/27$. 
We can apply the identity
 \begin{equation}
   {}_{2}F_{1}\!\left(c,1-c;\frac{1}{2};-z\right)=\frac{\left(\sqrt{1+z}+\sqrt{z}\right)^{2c-1}+\left(\sqrt{1+z}-\sqrt{z}\right)^{2c-1}}{2\sqrt{1+z}},
 \label{eq:id_2F1_prudnikov}
 \end{equation}
valid for $c\in(0,1)$ and $|z|<1$, see \cite[Eq.~7.3.3.4, p.~486]{prudnikov3_90}, to the right-hand side of~\eqref{eq:gener_h_m_2F1_4diag} with $c=2/3$ and deduce the formula for 
the Weyl $m$-function of the Jacobi operator $J(b)$ which reads
\[
 m_{b}(z)=-\frac{1}{z}\ {}_{2}F_{1}\!\left(\frac{1}{3},\frac{2}{3};\frac{1}{2};\frac{1}{z}\right)=-\frac{\left(\frac{\ii}{\sqrt{z}}+\sqrt{1-\frac{1}{z}}\right)^{1/3}+\left(-\frac{\ii}{\sqrt{z}}+\sqrt{1-\frac{1}{z}}\right)^{1/3}}{2\sqrt{z^{2}-z}},
\]
for $z\in\C\setminus[0,1]$. By evaluating the limit in~\eqref{eq:weyl-m_stieltjes_inversion}, one obtains 
\begin{equation}
 \frac{\dd\mu}{\dd x}(x)=\frac{\sqrt{3}}{4\pi}\frac{\left(1+\sqrt{1-x}\right)^{1/3}-\left(1-\sqrt{1-x}\right)^{1/3}}{x^{2/3}\sqrt{1-x}}, \quad x\in(0,1).
 \label{eq:density_4diag}
\end{equation}
This density appeared earlier in connection with Faber polynomials \cite{kuijlaars_jcam95}; see also \cite{coussementetal_tams08, duitskuijlaars_siamjmaa08}.

Next, we examine the operator $J(b)$ and the corresponding family of orthogonal polynomials in detail. First, we derive formulas for the diagonal sequence $\{b_{n}\}_{n=1}^{\infty}$
and off-diagonal sequence $\{a_{n}\}_{n=1}^{\infty}$ of $J(b)$. In \cite{egeciogluetal_dmtcsp08}, the Hankel determinant
\begin{equation}
 \det H_{n}=3^{n-1}\left(\prod_{i=0}^{n-1}\frac{(3i+1)(6i)!(2i)!}{(4i)!(4i+1)!}\right)a^{n(n-1)}
\label{eq:det_H_n_4diag}
\end{equation}
has been evaluated with $a=1$, see \cite[Eq.~(1)]{egeciogluetal_dmtcsp08}. The slightly more general identity~\eqref{eq:det_H_n_4diag} with the additional parameter $a$
can be justified by using the same argument as in the first paragraph of the proof of Lemma~\ref{lem:det_H_n_tridiag}. Moreover, the first equation from \cite[Eq.~(25)]{egeciogluetal_dmtcsp08} 
yields
\begin{equation} 
\det\tilde{H}_{n}=\frac{27n^{2}-8n-1}{2(4n-1)}\det H_{n}, \quad n\in\N,
\label{eq:det_tilde_H_n_4diag}
\end{equation}
(in the notation used in \cite[Eq.~(25)]{egeciogluetal_dmtcsp08}, $H_{1}\equiv H_{1}(n,0)$ coincides with $\det \tilde{H}_{n+1}$ and $H_{0}\equiv H_{0}(n,0)$ coincides with $\det H_{n+1}$).

By substituting the identities \eqref{eq:det_H_n_4diag} and \eqref{eq:det_tilde_H_n_4diag} in the general formulas \eqref{eq:a_n_det_Hankel} and \eqref{eq:b_n_det_Hankel}, one obtains
\[
 a_{1}^{2}=6a^{2} \; \mbox{ and } \; a_{k}^{2}=\frac{9(6k-5)(6k-1)(3k-1)(3k+1)}{4(4k-3)(4k-1)^{2}(4k+1)}a^{2}, \; \mbox{ for } k>1.
\]
and
\[
 b_{1}=3a \; \mbox{ and } \; b_{k}=\frac{3(36k^{2}-54k+13)}{2(4k-5)(4k-1)}a, \; \mbox{ for } k>1.
\]
The corresponding family of monic orthogonal polynomials generated by the recurrence~\eqref{eq:og_pol_recur}
is orthogonal with respect to the density~\eqref{eq:density_4diag} for $a=4/27$. Taking into account~\eqref{eq:og_rel} 
and the identity~\eqref{eq:det_H_n_4diag}, one gets the orthogonality relation
\[
 \int_{0}^{1}p_{n}(x)p_{m}(x)\rho(x)\dd x=(3-2\delta_{n,0})\left(\frac{4}{27}\right)^{\!2n}\frac{(3n+1)(6n)!(2n)!}{(4n)!(4n+1)!}\delta_{n,m}, \quad \forall m,n\in\N_{0}.
\]

The polynomial sequence $\{p_{n}\}_{n=0}^{\infty}$ does not coincide with any family listed in the Askey scheme \cite{koekoek10} either. Nevertheless, $p_{n}$ can be expressed as a linear combination
of the associated Jacobi polynomials introduced and studied by J.~Wimp in \cite{wimp_cjm87}. Following notation from \cite{wimp_cjm87}, the associated Jacobi polynomials $P_{n}^{(\alpha,\beta)}(x;c)$ constitute  
a three-parameter family of orthogonal polynomials generated by the same recurrence as the Jacobi polynomials $P_{n}^{(\alpha,\beta)}(x)$, but  every occurrence of $n$  in the 
coefficients of the recurrence relation defining the Jacobi polynomials is replaced by $n+c$, see \cite[Eq.~(12)]{wimp_cjm87}. Set
\[
 r_{n}^{(\alpha,\beta)}(x;c):=\frac{2^{n}(c+\alpha+\beta+1)_{n}(c+1)_{n}}{(2c+\alpha+\beta+1)_{2n}}P_{n}^{(\alpha,\beta)}(2x-1;c), \quad n\in\N_{0},
\]
and $r_{-1}^{(\alpha,\beta)}(x;c):=0$, where $(x)_{n}=x(x+1)\dots(x+n-1)$ is the Pochhammer symbol. Then, putting again $a=4/27$, one has
\begin{equation}
 2^{n}p_{n}(x)=r_{n}^{(\alpha,\beta)}(x;c)-\frac{4}{27}r_{n-1}^{(\alpha,\beta)}(x;c+1)-\frac{256}{729}r_{n-2}^{(\alpha,\beta)}(x;c+2), \quad n\in\N,
 \label{eq:p_n_rel_assoc_Jac}
\end{equation}
where $\alpha=1/2$, $\beta=-2/3$, and $c=-1/6$. Verification of~\eqref{eq:p_n_rel_assoc_Jac} is done in a completely routine way by showing that both sides
satisfy the same recurrence relation with the same initial conditions.

Finally, since Wimp derived the explicit formula for associated Jacobi polynomials in \cite[Theorem~1]{wimp_cjm87}, one can make use of this result together with~\eqref{eq:p_n_rel_assoc_Jac} 
to compute the explicit expression for $p_{n}(x)$. However, the computation is somewhat lengthy and the resulting formula is rather cumbersome; we omit the details and state only the final result
for the record. For $n\in\N_{0}$, one has
\[
 p_{n}(x)=A_{n}\sum_{k=0}^{n}B_{n}(k)C_{n}(k)x^{k},
\]
where
\[
 A_{n}=(-1)^{n}\frac{\left(\frac{1}{6}\right)_{n}\left(\frac{5}{6}\right)_{n}}{5\left(n+\frac{1}{2}\right)_{n}n!},
 \quad
 B_{n}(k)=\frac{\left(-n\right)_{k}\left(n+\frac{1}{2}\right)_{k}}{\left(\frac{1}{6}\right)_{k}\left(\frac{5}{6}\right)_{k}},
\]
and
\[
 C_{n}(k)=\sum_{i=0}^{n-k}\frac{\left(k-n\right)_{i}\left(n+k+\frac{1}{2}\right)_{i}\left(-\frac{5}{6}\right)_{i}\left(-\frac{1}{6}\right)_{i}}
 {\left(-\frac{1}{2}\right)_{i}\left(k+\frac{1}{6}\right)_{i}\left(k+\frac{5}{6}\right)_{i}}
 \frac{(6i-5)(18i+1)}{(2i-1)(2i+1)}.
\]

\subsection{Example 3}\label{subsec:ex3}

Both cases treated in the previous subsections, where the limiting measure was derived fully explicitly, can be thought of  as special cases 
of the more general symbol
\begin{equation}
 b(z)=\frac{1}{z^{r}}(1+az)^{r+s},
\label{eq:def_b_gener_ex}
\end{equation}
(up to a shift by a constant term), where $r,s,\in\N$ and $a\in\R\setminus\{0\}$. By using the binomial formula, one computes
\[
 h_{m}=\frac{1}{2\pi}\int_{0}^{2\pi}b^{m}\left(e^{\ii\theta}\right)\dd\theta=\binom{(r+s)m}{rm}a^{rm}, \quad m\in\N_{0}.
\]

First, we prove that $b\in\mathscr{R}$ for all $r,s\in\N$ and $a\in\R\setminus\{0\}$ by showing that $b^{-1}(\R)$ contains a Jordan curve. Without loss of generality, 
we may again put $a=1$ (otherwise one can take the $1/a$ multiple of the curve $\gamma$ given below). Define
\begin{equation}
 \gamma(t):=\frac{\sin\omega t}{\sin(1-\omega)t}e^{\ii t}, \quad t\in(-\pi,\pi],
\label{eq:def_gamma_spec_exam}
\end{equation}
where $\omega:=r/(r+s)$. The value $\gamma(0)=r/s$ is determined by the corresponding limit. Figure~\ref{fig:jordan_spec_ex} shows the Jordan $\gamma$ for three special choices of $r$ and $s$.
Expressing the sine function in terms of complex exponentials, one easily
verifies that
\begin{equation}
 b(\gamma(t))=\frac{\sin^{r+s} t}{\sin^{r}\!\left(\omega t\right)\sin^{s}\!\left((1-\omega)t\right)}, \quad t\in(-\pi,\pi],
\label{eq:b_comp_gamma_spec_exam}
\end{equation}
where the value at $t=0$ is the respective limit. Thus, $b$ restricted to the image of $\gamma$ is real-valued and Theorem~\ref{thm:summary} implies $b\in\mathscr{R}$.

The parametrization \eqref{eq:def_gamma_spec_exam} has the polar form~\eqref{eq:gamma_polar_par}. Moreover, $b$ has no critical point on the arc of $\gamma$ in the upper half-plane since
\[
 b'(\gamma(t))=\frac{s\gamma(t)-r}{(1+\gamma(t))\gamma(t)}b(\gamma(t))\neq 0, \quad \forall t\in(0,\pi),
\]
as one readily verifies. Taking  into account that 
\[
 b(\gamma(0))=\frac{(r+s)^{r+s}}{r^{r}s^{s}} \quad \mbox{ and } \quad b(\gamma(\pi))=0,
\]
we see that $b\circ\gamma$ is strictly decreasing on $(0,\pi)$. Consequently, by Theorem~\ref{thm:main_suff}, the limiting mesure is supported on the interval
\begin{equation}
 \supp\mu=\left[0,r^{-r}s^{-s}(r+s)^{r+s}\right]
 \label{eq:supp_mu_spec_exam}
\end{equation}
and, according to~\eqref{eq:mu_dist_func_simple_non_invert}, the distribution function of $\mu$ satisfies
\[
 F_{\mu}\left(b(\gamma(t))\right)=1-\frac{t}{\pi}, \quad \mbox{ for }\; 0\leq t\leq \pi.
\]

It seems that, for general parameters $r$ and $s$, the density of $\mu$ can not be expressed explicitly because the function in~\eqref{eq:b_comp_gamma_spec_exam} cannot be inverted;
see Figure~\ref{fig:density_spec_ex} for numerical plots of the density of $\mu$. Let us point out that, besides the cases $r=1$
and either $s=1$ or $s=2$, one can also derive explicitly the density for the self-adjoint 5-diagonal case, i.e.  $r=s=2$ ($a=1$). The resulting formula reads 
\[
 \rho(x)=\frac{\sqrt{4+\sqrt{x}}}{2\pi x^{3/4}\sqrt{16-x}}, \quad x\in(0,16),
\]
since the  Weyl $m$-function can be expressed as
\[
 m_{b}(z)=-\frac{1}{z}\sqrt{\frac{1+\sqrt{1-16/z}}{2(1-16/z)}}\!, \quad z\in\C\setminus[0,16].
\]
The above formula follows from the generating function of $\{h_{m}\}_{m=0}^{\infty}$ which, for $|z|<1/16$, reads
\begin{align*}
 \sum_{m=0}^{\infty}\binom{4m}{2m}z^{m}&={}_{2}F_{1}\left(\frac{1}{4},\frac{3}{4};\frac{1}{2};16z\right)=
 \frac{\sqrt{\sqrt{1-16z}+4\ii\sqrt{z}}+\sqrt{\sqrt{1-16z}-4\ii\sqrt{z}}}{2\sqrt{1-16z}},\\
 &=\sqrt{\frac{1+\sqrt{1-16z}}{2(1-16z)}},
\end{align*}
where we make use of the identity~\eqref{eq:id_2F1_prudnikov} with $c=3/4$.

Derivation of a closed formula for the sequences $\{a_{n}\}_{n=1}^{\infty}$ and $\{b_{n}\}_{n=1}^{\infty}$ determining the Jacobi operator $J(b)$ for general values of $r$ and $s$, 
seems to be out of our reach. Even a closed expression for the determinant of the Hankel matrix $H_{n}$ determined by the sequences
 \[
  h_{m}=\binom{4m}{m} \quad\mbox{ or }\quad h_{m}=\binom{4m}{2m},
 \]
corresponding to the special cases with $a=1$, $r=1$ and $s=3$ or $r=2$ and $s=2$, respectively, is unknown, to the best of our knowledge. 
 In fact, numerical experiments with the sequence of these determinants show 
a presence of huge prime factors which might indicate that no closed formula, similar to~\eqref{eq:det_H_n_4diag}, exists for the Hankel 
determinants. Nevertheless, Theorem~\ref{thm:compact} and equation~\eqref{eq:supp_mu_spec_exam} yield the limit formulas:
\[
 2\lim_{n\to\infty} a_{n}=\lim_{n\to\infty} b_{n}=\frac{(r+s)^{r+s}}{2r^{r}s^{s}},
\]
for all $r,s\in\N$.

\subsection{More general examples based on Example 3}

Most of the results of this paper are devoted to banded Toeplitz matrices but Theorem~\ref{thm:gamma_impl_real_spec} is applicable to matrices with more general symbols.
A combination of the previous example with the symbol \eqref{eq:def_b_gener_ex} and Theorem~\ref{thm:gamma_impl_real_spec} provides us with more examples of possibly 
non-self-adjoint Toeplitz matrices given by a more general symbol. For instance, one can proceed as follows. 

If $f$ is a function analytic on $\C\setminus\{0\}$ which maps $\R\setminus\{0\}$ to $\R$ and $b$ is as in~\eqref{eq:def_b_gener_ex}, then the symbol $a:=f\circ b$ satisfies
the assumptions of Theorem~\ref{thm:gamma_impl_real_spec} with the Jordan curve given by \eqref{eq:def_gamma_spec_exam}. Consequently, $\spec T_{n}(a)\subset\R$ for all $n\in\N$.
To be even more concrete, the possible choices of $f$ comprise, e.g., $f\in\{\exp,\sin,\cos,\sinh,\cosh,\dots\}$, producing many examples of non-self-adjoint and non-banded Toeplitz matrices 
whose principal submatrices have purely real spectrum.

%
%

\subsection{Various numerical experiments}\label{subsec:var_numer}

To illustrate a computational applicability of Theorems~\ref{thm:main_suff_generalized} and~\ref{thm:main_suff}, we add two more complicated examples treated numerically 
using Wolfram Mathematica. In particular, to emphasize the connection between the existence of a Jordan curve in $b^{-1}(\R)$ and the reality of the limiting set
of eigenvalues of the corresponding Toeplitz matrices, we provide some numerical plots on this account.

\subsubsection{Example 4}
We plot the numerically obtained  density  by applying Theorem~\ref{thm:main_suff} to the symbol 
\[
 b(z)=\frac{1}{z^3}-\frac{1}{z^2}+\frac{7}{z}+9z-2z^2+2z^3-z^4,
\]
whose net is shown in Figure~\ref{fig:net_illust_yes_no} (left). 

Although we can only check the validity of  the assumptions of Theorem~\ref{thm:main_suff}  numerically, looking at Figure~\ref{fig:net_illust_yes_no} it is reasonable to believe the Jordan curve present in $b^{-1}(\R)$ 
can be parametrized by polar coordinates \eqref{eq:gamma_polar_par}. Further, a numerical computation indicates that no critical point of $b$ lies on the Jordan curve in the upper half-plane.
The real critical points of $b$ closest to the origin are $z_{1}\approx1.077904$ and $z_{2}\approx-0.844126$ and the corresponding critical values are $b(z_{1})\approx14.9641$ and 
$b(z_{2})\approx-22.0915$. So the limiting measure $\mu$ should be supported approximately on the interval $[-22.09,14.96]$ which is in agreement with the numerical results  obtained by 
an implementation of the algorithm for the  computation of  $\Lambda(b)$ suggested in~\cite{beamwarming_siamjsc93}.
The density of $\mu$ is plotted in Figure~\ref{fig:density_ex4}.

\subsubsection{Example 5} For an illustration of Theorem~\ref{thm:main_suff_generalized} which allows to have non-real critical points of~$b$ on the arc of Jordan curve in $b^{-1}(\R)$, 
we consider the symbol
\[
 b(z)=\frac{1}{z^{3}}+\frac{1}{z^{2}}+\frac{1}{z}+z+z^{2}+z^{3}.
\]
The corresponding Toeplitz matrix $T(b)$ is self-adjoint and hence the Jordan curve in $b^{-1}(\R)$ is the unit circle. The function $b$ has two critical points with positive imaginary part
on the unit circle, namely
\begin{align*}
 z_{1}&:=\frac{1}{6}\left(-1 + \sqrt{7} + \ii\sqrt{2(14 + \sqrt{7})}\right)\approx 0.274292 + 0.961646\ii,\\
 z_{2}&:=\frac{1}{6}\left(-1 - \sqrt{7} + \ii\sqrt{2(14 - \sqrt{7})}\right)\approx -0.607625 + 0.794224\ii
\end{align*}
and the corresponding values are
\[
 b(z_{1})\approx-2.63113 \quad \mbox{ and } \quad b(z_{2})\approx0.112612,
\]
see Figure \ref{fig:net_ex5}. Theorem~\ref{thm:main_suff_generalized} tells us that each arc of the unit circle between two critical points gives rise to a measure. Labeling the 
measures in agreement with Theorem~\ref{thm:main_suff_generalized} (starting at $1$ and traversing the arcs of the unit circle in the counterclockwise direction) we get measures 
$\mu_{1}$, $\mu_{2}$, and $\mu_{3}$, which, since $b(1)=6$ and $b(-1)=-2$, are supported approximately on the intervals $[-2.63, 6]$, $[-2.63,0.11]$, and $[-2,0.11]$, respectively. 
The limiting measure then equals $\mu=\mu_{1}+\mu_{2}+\mu_{3}$. The illustration of the corresponding densities are given in Figures~\ref{fig:densities_ex5} and ~\ref{fig:density_ex5}.

The graph of the density of $\mu$ suggests that the eigenvalues of $T_{n}(b)$ cluster with higher density to  the left of the point $0.11$ which has also been observed numerically, see Figure~\ref{fig:eigen_plot_ex5}.

\subsubsection{Breaking the reality of $\Lambda(b)$}
Our final plots are devoted to an illustration of the connection between the presence of a Jordan curve in $b^{-1}(\R)$ and the reality of $\Lambda(b)$. We go back to the Example~4 once again  
and introduce an additional parameter $\alpha\in\R$ at $z^{2}$ getting the symbol
\[
 b(z)=\frac{1}{z^3}-\frac{1}{z^2}+\frac{7}{z}+9z+\alpha z^2+2z^3-z^4.
\]

In Figure~\ref{fig:break_ex}, it is shown that, as the parameter $\alpha$ increases from $-2$ to $2$, the Jordan curve in $b^{-1}(\R)$ gets ``destroyed'' by an incoming curve from the net of $b$.
At the same time, $\Lambda(b)$ changes from a real interval to a set with non-real values. The set $\Lambda(b)$ is plotted by using the algorithm from~\cite{beamwarming_siamjsc93}.


\section*{Acknowledgments}

The authors wish to express their gratitude to Professor A.~B.~J.~Kuijlaars for pointing out an error in the first version of the manuscript
and to Professor M.~Duits for his interest in the topic.

\section*{Appendix - a determinant formula}

The following determinant formula presented in Lemma~\ref{lem:det_H_n_tridiag} was used in Subsection~\ref{subsec:ex1}. It can be deduced from the well-known identity for the Hankel transform of the sequence of central binomial
coefficients. Although the derivation is elementary, it is not completely straightforward and therefore we provide it with its proof for the convenience of a reader.

\begin{lem}\label{lem:det_H_n_tridiag}
 Let $n\in\N$ and $H_{n}$ be the Hankel matrix of the form~\eqref{eq:hankel_matrix} with elements given by~\eqref{eq:h_m_tridiag}, then one has
 \[
  \det H_{n}=2^{n-1}a^{n(n-1)/2}.
 \]
\end{lem}

\begin{proof}
 First note that, if $H_{n}$ and $G_{n}$ are two Hankel matrices with $(H_{n})_{i,j}=h_{i+j-2}$ and $(G_{n})_{i,j}=g_{i+j-2}$ such that $h_{m}=\alpha^{m} g_{m}$, for some $\alpha\in\C$, 
 then $H_{n}=D_{n}(\alpha)G_{n}D_{n}(\alpha)$ where $D_{n}(\alpha)=\diag(1,\alpha,\dots,\alpha^{n-1})$. Therefore it suffices to verify the statement for $a=1$.

 Let $k,n\in\N$ and $G_{n}^{(k)}$ be the Hankel matrix whose entries are given by
 \[
  \left(G_{n}^{(k)}\right)_{i,j}=\binom{2(i+j+k-2)}{i+j+k-2}, \quad  \mbox{ for } i,j\in\{1,2,\dots, n\}.
 \]
 The formula for $\det G_{n}^{(k)}$ is well-known and, in particular, one has 
 \begin{equation}
  \det G_{n}^{(1)}=2\det G_{n}^{(0)}=2^{n}, \quad \forall n\in\N,
 \label{eq:hank_det_centr_binom_01}
 \end{equation}
 see, for example, \cite{aigner_jctsa99, armassethuraman_jis08}.
 
 Next we make use of the direct sum decomposition $\C^{n}=\C_{\text{odd}}^{n}\oplus \C_{\text{even}}^{n}$ where
 \[
  \C_{\text{odd}}^{n}:=\spn\{e_{2j-1} \mid 1\leq 2j-1 \leq n\}  \quad \mbox{ and } \quad \C_{\text{even}}^{n}:=\spn\{e_{2j} \mid 1\leq 2j \leq n\}.
 \]
 Since $h_{2m-1}=0$, for all $m\in\N$, both subspaces
 $\C_{\text{odd}}^{n}$ and $\C_{\text{even}}^{n}$ are $H_{n}$-invariant, i.e., $H_{n}\C_{\text{odd}}^{n}\subset\C_{\text{odd}}^{n}$ and $H_{n}\C_{\text{even}}^{n}\subset\C_{\text{even}}^{n}$.
 The matrix $H_{n}$ decomposes accordingly as
 \[
  H_{n}=H_{n}^{\text{odd}}\oplus H_{n}^{\text{even}},
 \]
 where
 \[
  H_{n}^{\text{odd}}=G_{\lfloor(n+1)/2\rfloor}^{(0)} \quad \mbox{ and } \quad H_{n}^{\text{even}}=G_{\lfloor n/2\rfloor}^{(1)}.
 \]
 Thus, by using formulas in~\eqref{eq:hank_det_centr_binom_01}, one obtains
 \[
  \det H_{n}=\det\left(G_{\lfloor(n+1)/2\rfloor}^{(0)}\right)\det\left(G_{\lfloor n/2\rfloor}^{(1)}\right)=2^{n-1}.
 \]
\end{proof}

\newpage

\begin{figure}[h]
	\includegraphics[width= 0.48 \textwidth]{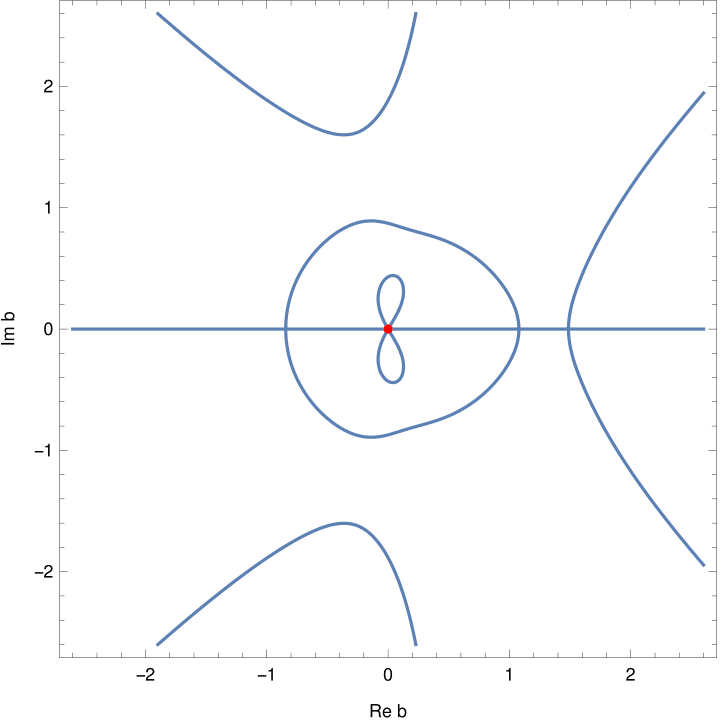} \quad 
	\includegraphics[width= 0.48 \textwidth]{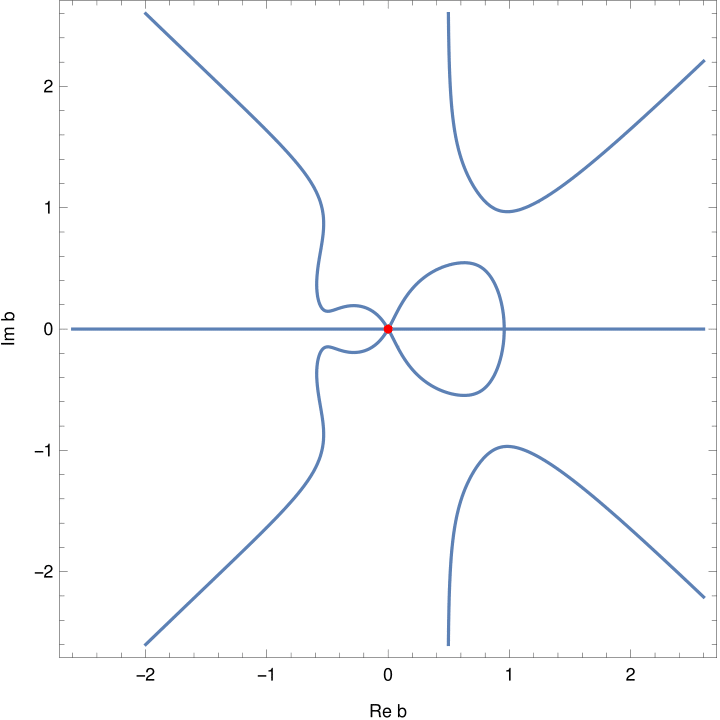}
	\caption{The plots of the net of $b(z)=1/z^3-1/z^2+7/z+9 z-2 z^2+2 z^3-z^4$ (left) and 
	$b(z)=-2/z^3-4/z^2+12/z+8 z-10 z^2+8 z^3-4 z^4$ (right). The red dot designates the origin. 
	On the left-hand side, $b^{-1}(\R)$ contains a Jordan curve, while this is not the case for the
	example plotted on the right-hand side.}
	\label{fig:net_illust_yes_no}
\end{figure}

\begin{figure}[h]
	\includegraphics[width= 0.9 \textwidth]{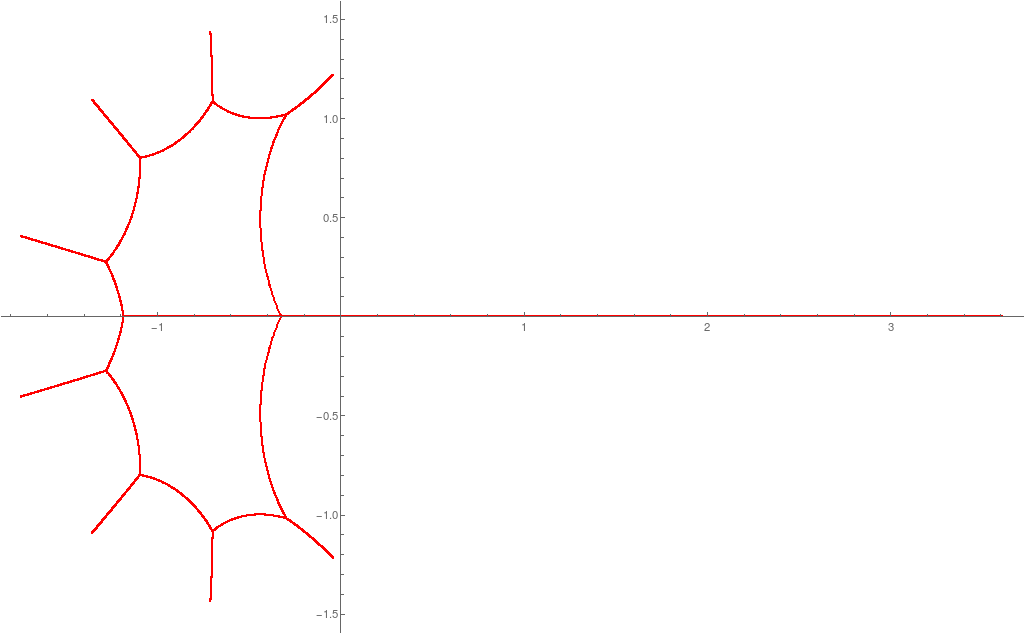}
	\caption{The set $\Lambda(b)$, for $b(z)=1/z+\sum_{k=1}^{10}kz^{k}$, which seems to separate the plane.}
	\label{fig:separ}
\end{figure}

\begin{figure}[h]
	\includegraphics[width= 0.7\textwidth]{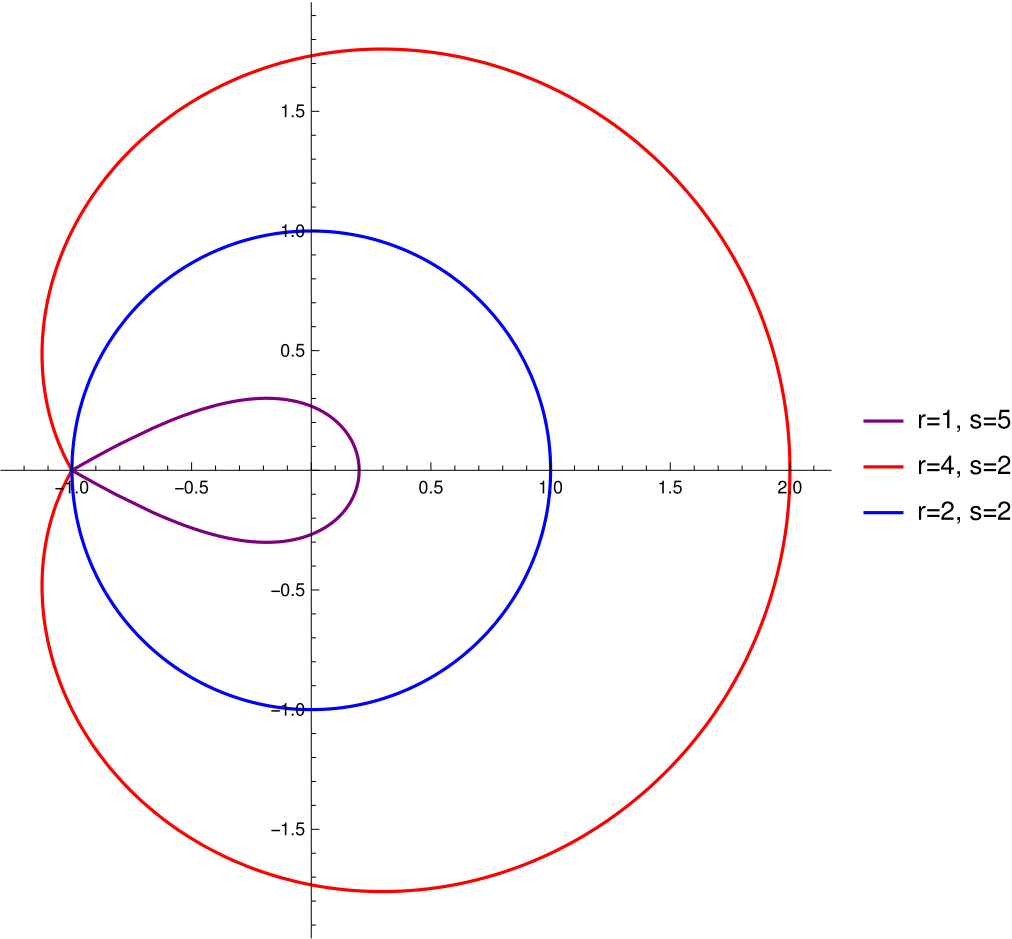}
	\caption{Illustration for Example~3: The plots of the Jordan curves $\gamma$ given by~\eqref{eq:def_gamma_spec_exam} 
	for 3 particular choices of parameters $r$ and $s$.}
	\label{fig:jordan_spec_ex}
\end{figure}

\begin{figure}[h]
	\includegraphics[width= 0.31 \textwidth]{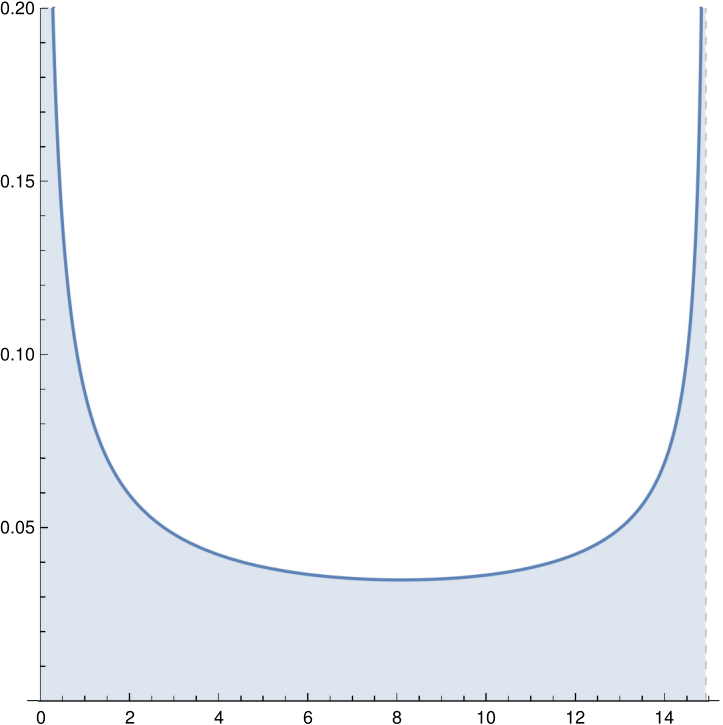} \quad 
	\includegraphics[width= 0.31 \textwidth]{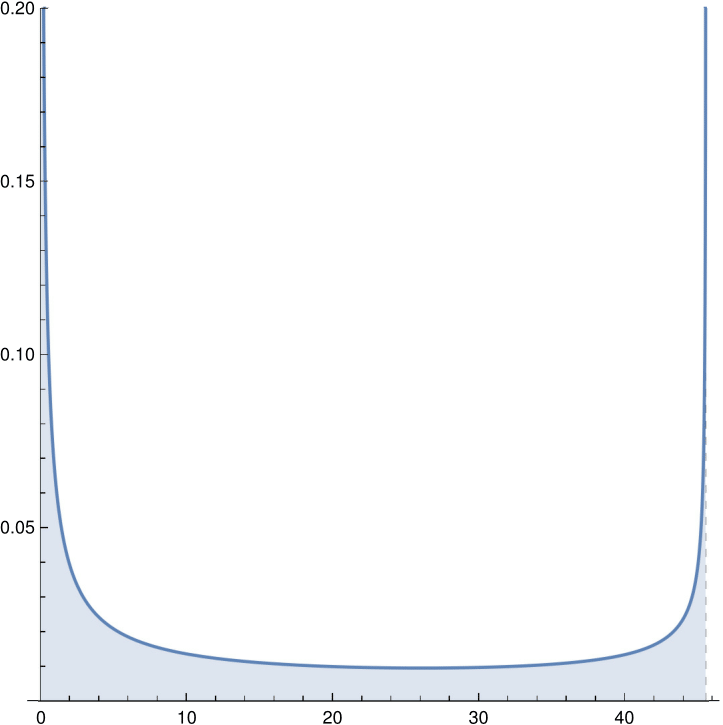} \quad
	\includegraphics[width= 0.31 \textwidth]{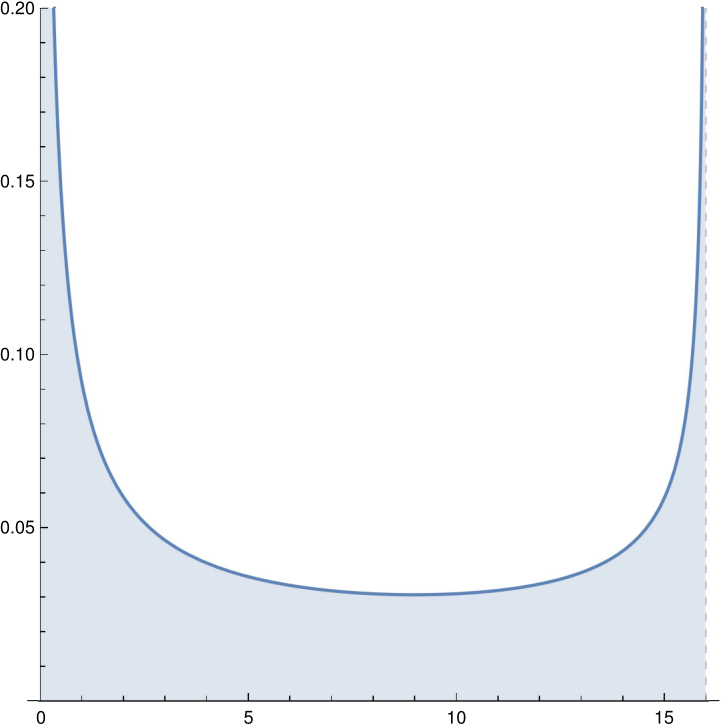}
	\caption{Illustration for Example~3: The plots of the density of $\mu$ for $r=1$, $s=5$ (left), $r=4$, $s=2$ (center), and $r=s=2$ (right).}
	\label{fig:density_spec_ex}
\end{figure}

\begin{figure}[h]
	\includegraphics[width= 0.8\textwidth]{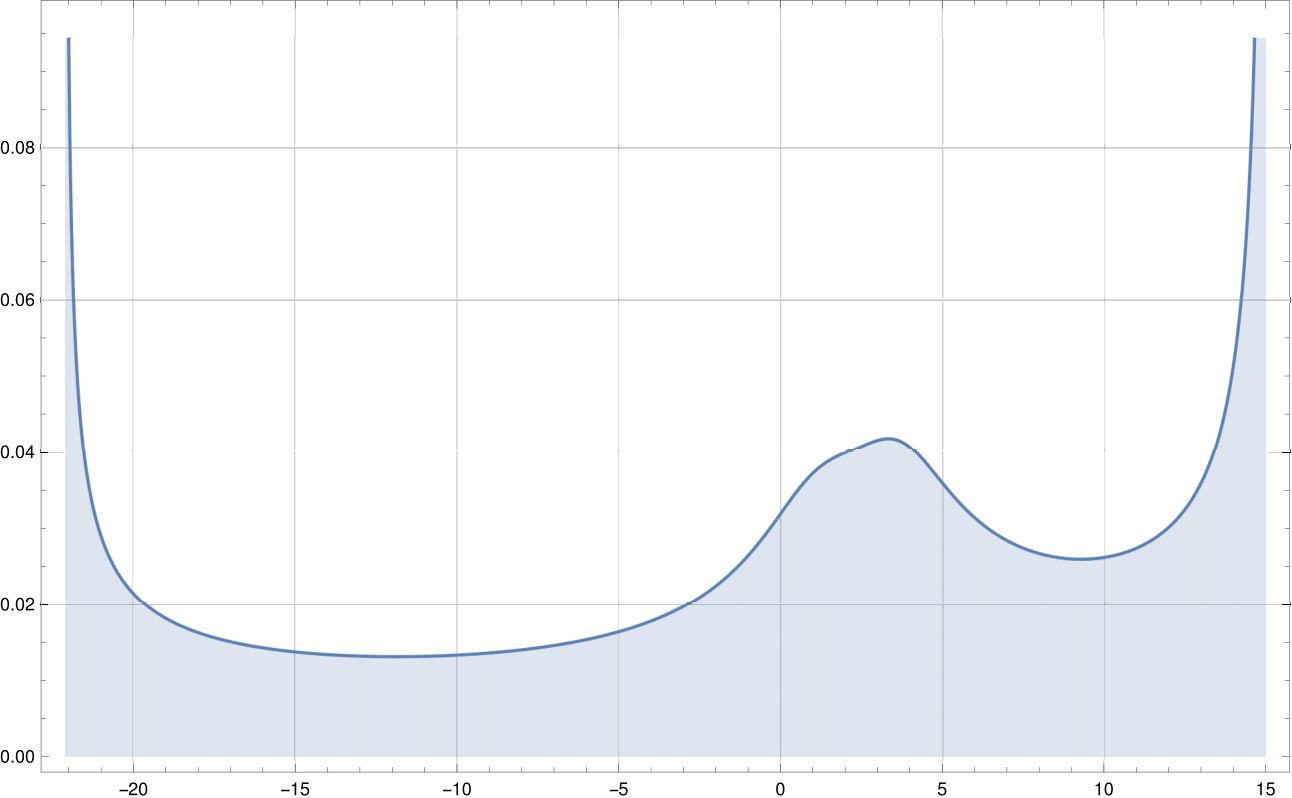}
	\caption{Illustration for Example~4: The plot of the density of $\mu$.}
	\label{fig:density_ex4}
\end{figure}

\begin{figure}[h]
	\includegraphics[width= 0.5\textwidth]{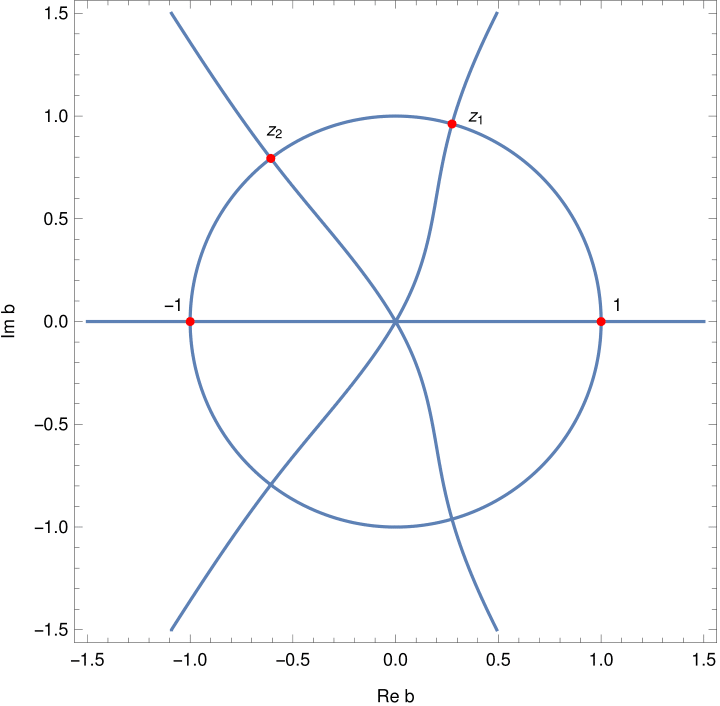}
	\caption{Illustration for Example~5: The plot of the net of $b(z)=1/z^3+1/z^2+1/z+z+z^2+z^3$.}
	\label{fig:net_ex5}
\end{figure}

\begin{figure}[h]
	\flushright\includegraphics[width= 0.9 \textwidth]{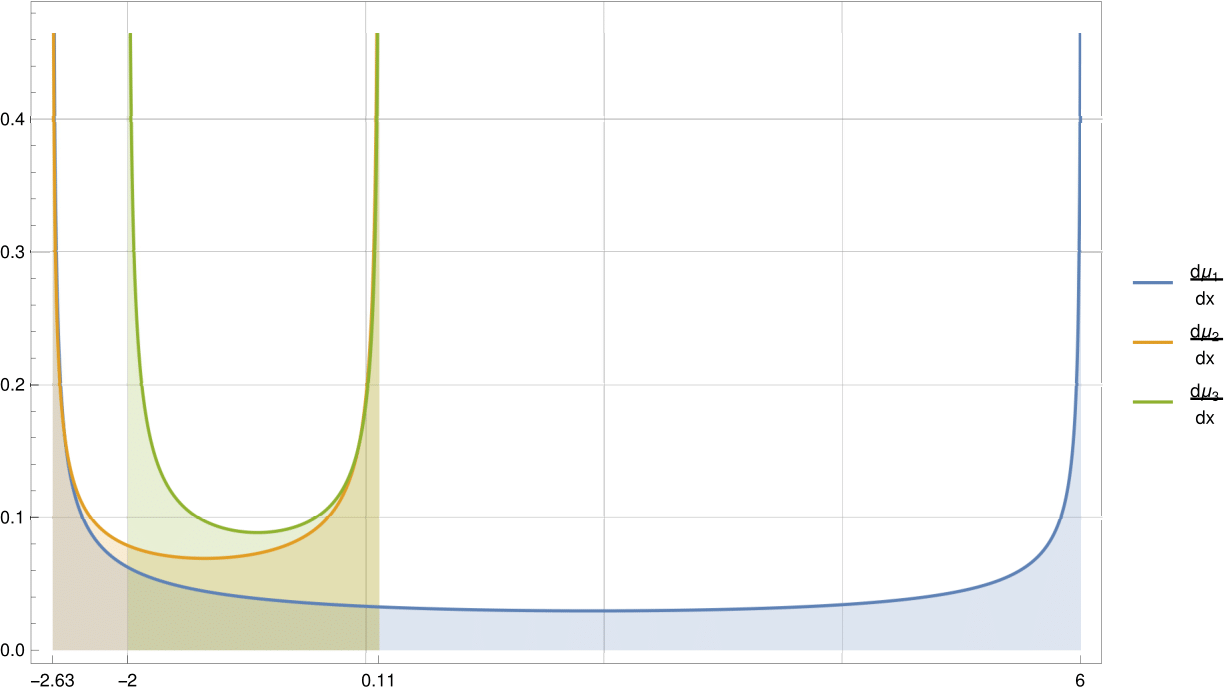}
	\caption{Illustration for Example~5: The plots of densities of $\mu_{1}$, $\mu_{2}$, and $\mu_{3}$.}
	\label{fig:densities_ex5}
\end{figure}

\begin{figure}[h]	
	\includegraphics[width= 0.9 \textwidth]{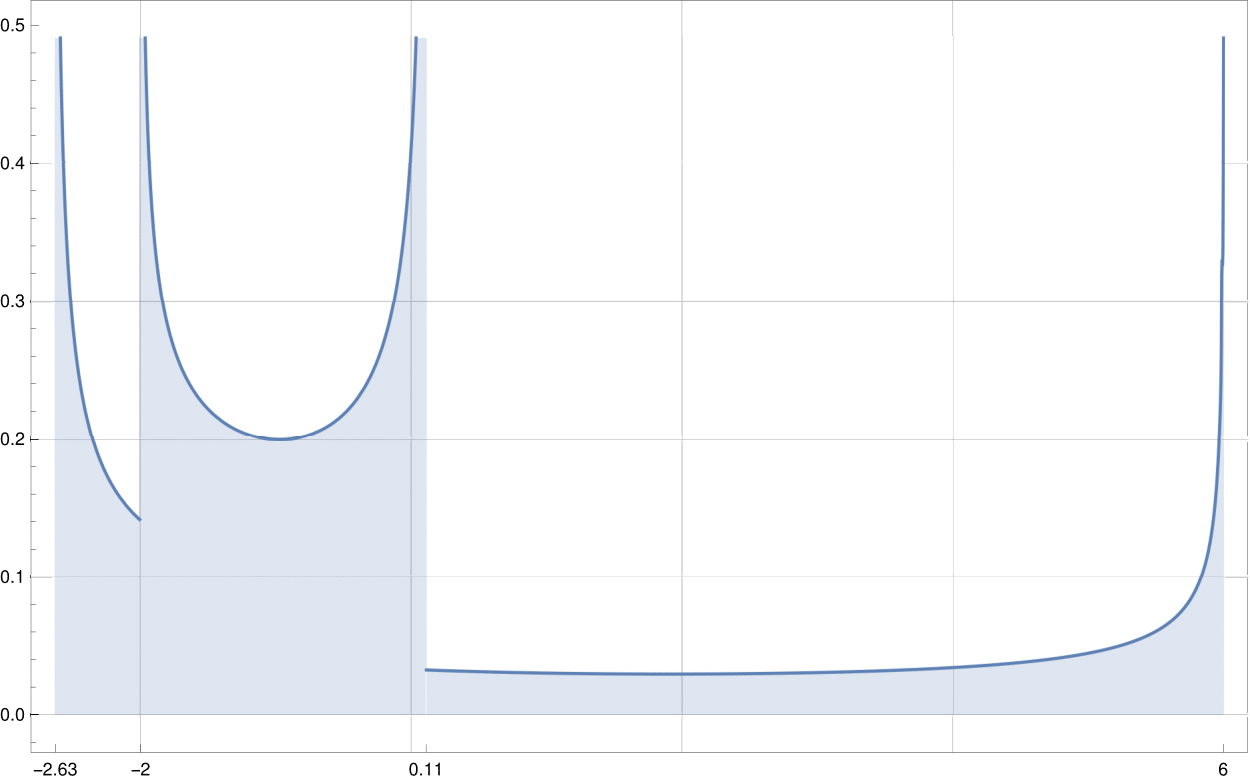}
	\caption{Illustration for Example~5: The plot of the density of $\mu$.}
	\label{fig:density_ex5}
\end{figure}

\begin{figure}[h]
	\includegraphics[width= 0.98 \textwidth]{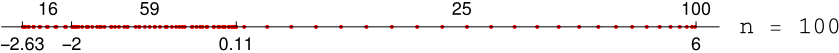}
	\vskip12pt
	\includegraphics[width= 0.98 \textwidth]{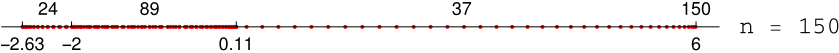} 
	\vskip12pt
	\includegraphics[width= 0.98 \textwidth]{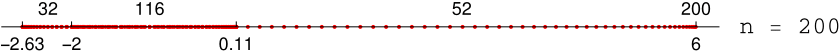}
	\caption{Illustration for Example~5: The distribution of eigenvalues of $T_{n}(b)$ for $n\in\{100,150,200\}$ in the interval $[-2.63,6]$. The numbers above the segments
	indicate the number of eigenvalues within the respective segment.}
	\label{fig:eigen_plot_ex5}
\end{figure}

\begin{figure}[h]
	\includegraphics[width= 0.19 \textwidth]{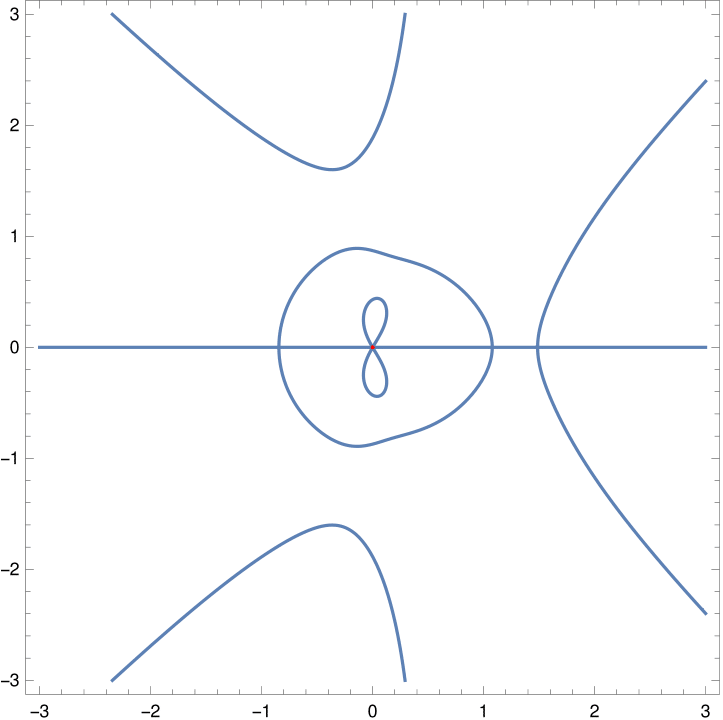} 
	\includegraphics[width= 0.19 \textwidth]{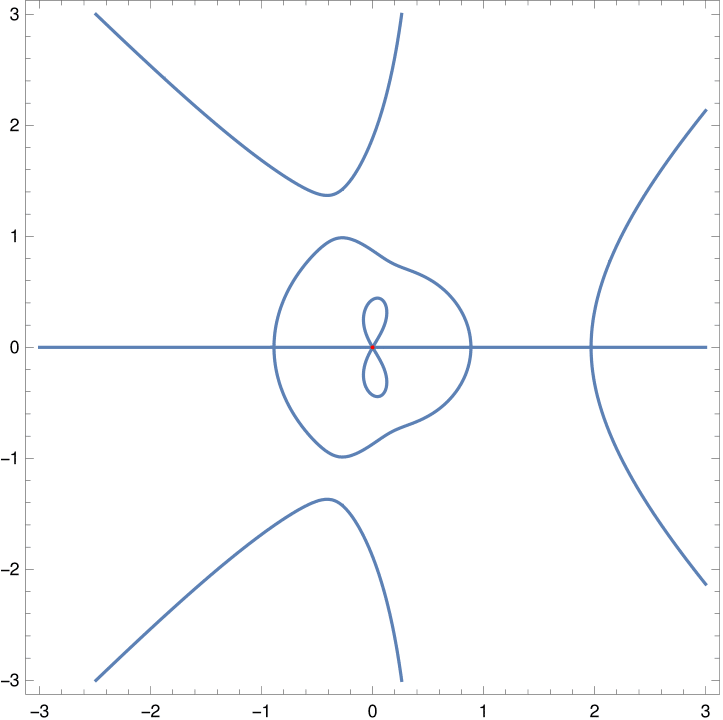} 
	\includegraphics[width= 0.19 \textwidth]{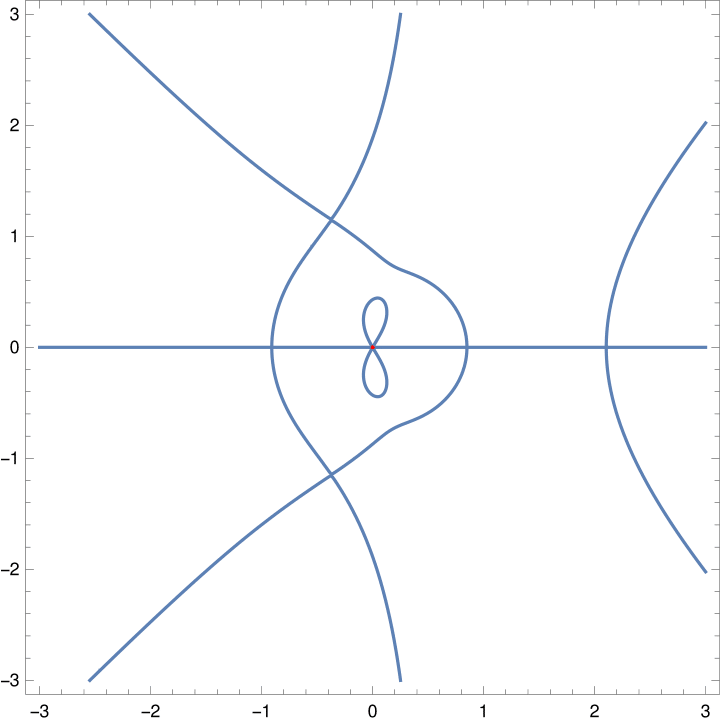} 
	\includegraphics[width= 0.19 \textwidth]{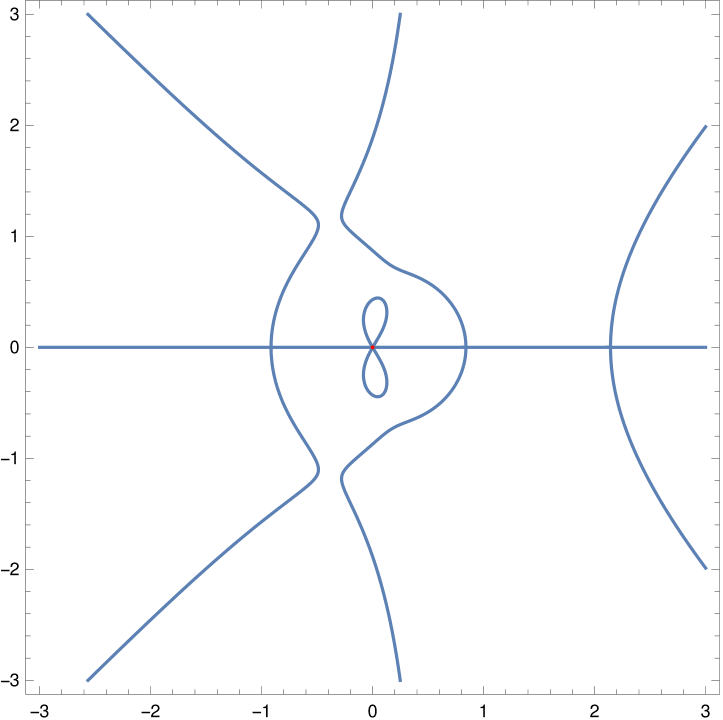}
	\includegraphics[width= 0.19 \textwidth]{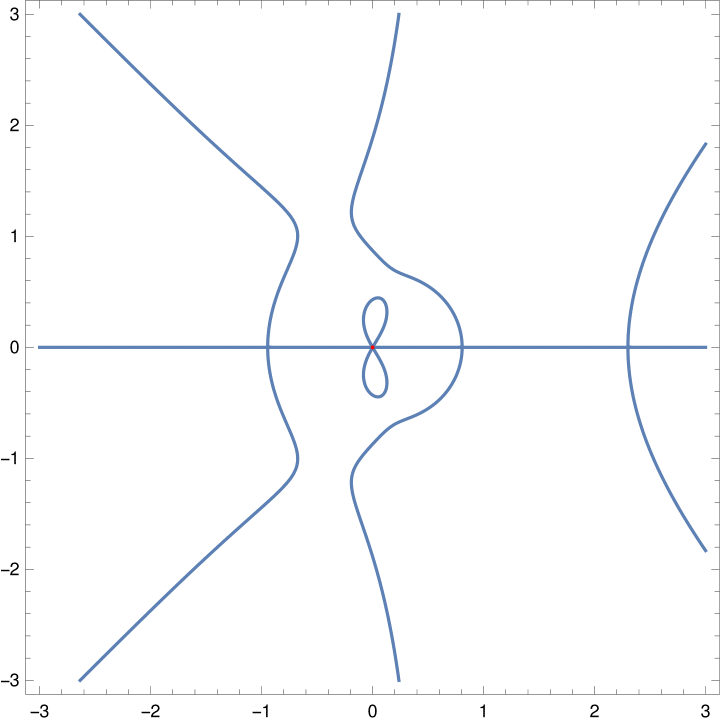}
	\vskip12pt
	\includegraphics[width= 0.19 \textwidth]{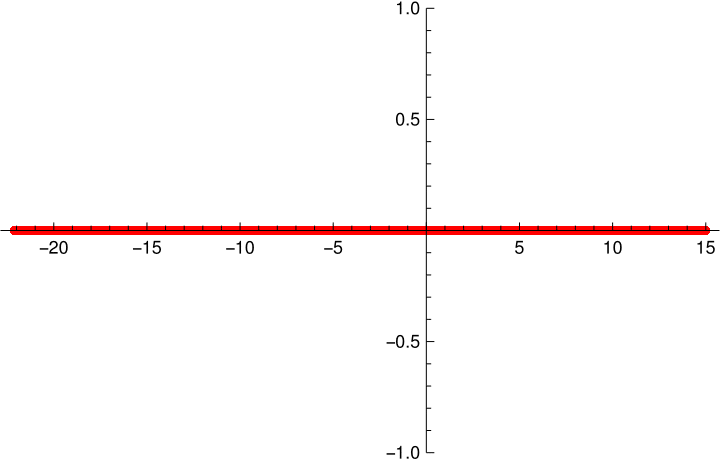}
	\includegraphics[width= 0.19 \textwidth]{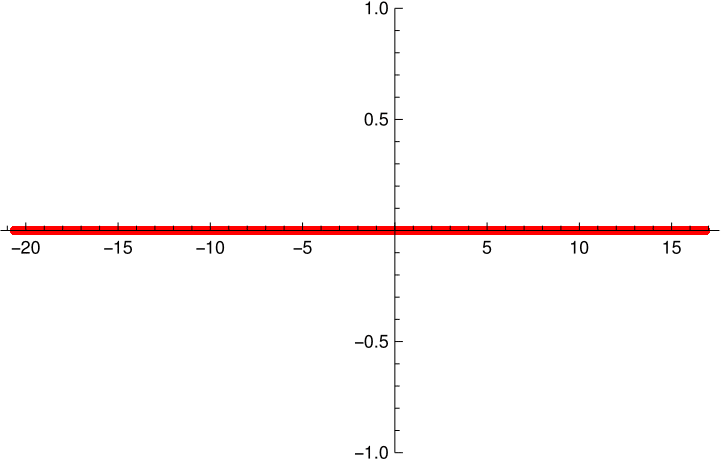} 
	\includegraphics[width= 0.19 \textwidth]{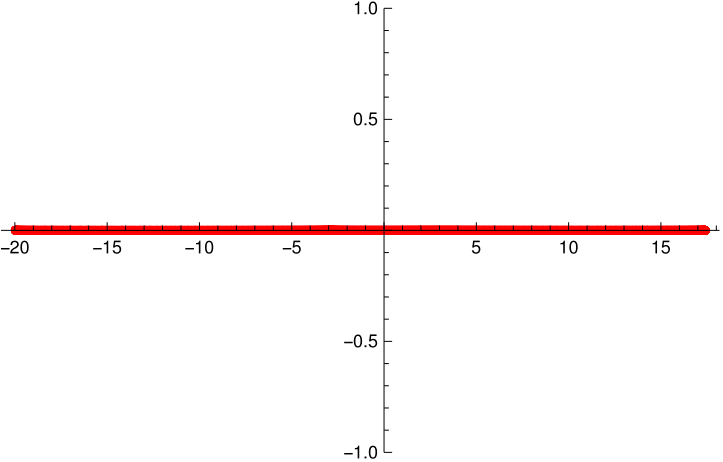} 
	\includegraphics[width= 0.19 \textwidth]{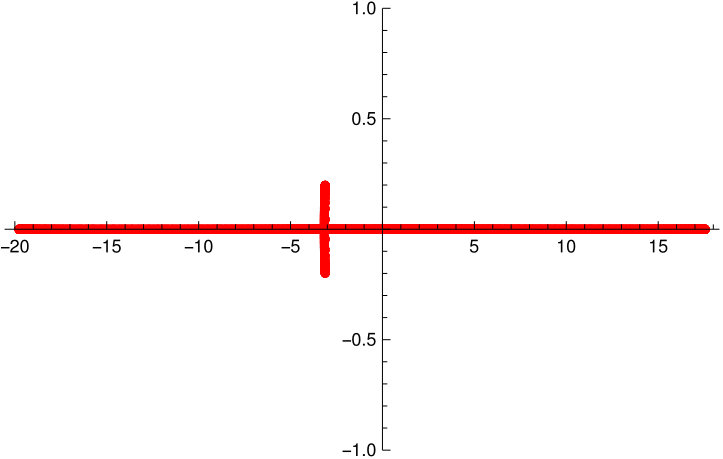} 
	\includegraphics[width= 0.19 \textwidth]{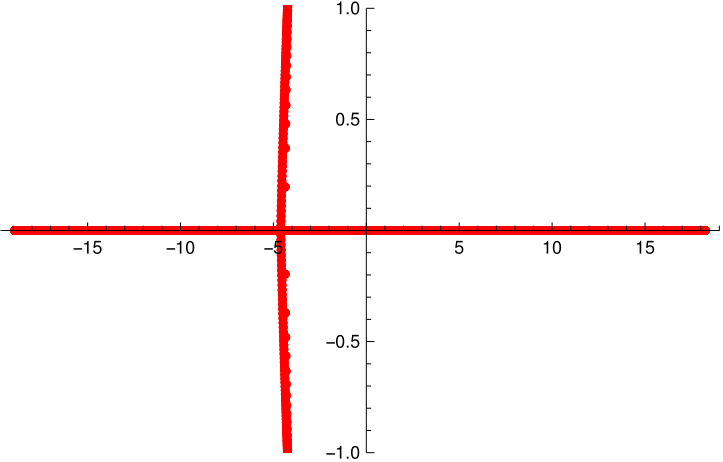} 
	\caption{The plots of the net of $b(z)=1/z^3-1/z^2+7/z+9 z+\alpha z^2+2 z^3-z^4$ (above) and 
	corresponding plots of $\Lambda(b)$ (below) for $\alpha\in\{-2,0,0.77,1,2\}$.}
	\label{fig:break_ex}
\end{figure}

$ $
\newpage 
$ $
\newpage
$ $
\newpage
$ $
\newpage
$ $
\newpage

\section*{Errata to "Non-self-adjoint Toeplitz matrices whose principal submatrices have real spectrum"}

\date{\today}

\begin{abstract}
We announce an error in the proof of Theorem~8 of this article which has been published in \emph{Constr.} \emph{Approx.} \textbf{48}(2) (2019) 191--226.
\end{abstract}

\maketitle

In the proof of Theorem~\ref{thm:gamma_impl_real_spec}, the false equality
\[
 \langle u,T_{n}(\overline{a})u\rangle=\langle f_{u},\overline{a}f_{u}\rangle_{\gamma}
\]
has been used. If the terms are interpreted using the notation from Section~\ref{sec:main}, the right-hand side equals the integral
\[
 \frac{1}{2\pi\ii}\int_{-\pi}^{\pi}\overline{a(\gamma(t))}f_{u}(\gamma(t))\overline{f_{u}(\gamma^{*}(t))}\frac{\dot{\gamma}(t)}{\gamma(t)}\dd t,
\]
while the left-hand side coincides with $\langle u,T_{n}^{*}(a)u\rangle$ and can be expressed as the integral
\[
 \frac{1}{2\pi\ii}\int_{-\pi}^{\pi}\overline{a(\gamma^{*}(t))}f_{u}(\gamma(t))\overline{f_{u}(\gamma^{*}(t))}\frac{\dot{\gamma}(t)}{\gamma(t)}\dd t.
\]
The two integrals do not coincide in general. Unfortunately, this error is an essential problem for the idea of the proof of Theorem~8, that was based on a contour integral representation of the quadratic form of a Toeplitz matrix $T_{n}(a)$, where the integration path is chosen as the Jordan curve $\gamma$ on which the symbol $a$ is real-valued.

Recall that Theorem~\ref{thm:summary}, which is the main result of the paper, claims that the following 3~statements are equivalent:
\begin{enumerate}
\item[(i)] $\Lambda(b)\subset\R$,
\item[(ii)] $b^{-1}(\R)$ contains a Jordan curve,
\item[(iii)] $\spec(T_{n}(b))\subset\R$ for all $n\in\N$,
\end{enumerate}
where $b$ is a Laurent polynomial, $T_{n}(b)$ the $n\times n$ Toeplitz matrix given by the symbol $b$, and $\Lambda(b)$ is the set of limit points of eigenvalues of $T_{n}(b)$, as $n\to\infty$. The logic of the proof of Theorem~1 was to establish implications (i)$\Rightarrow$(ii)$\Rightarrow$(iii)$\Rightarrow$(i). The currently unproven Theorem~8 was used to prove (ii)$\Rightarrow$(iii). As we are not able to find an alternative proof at the moment, the implication (ii)$\Rightarrow$(iii) remains unproven. 

However, we strongly believe that the currently unproven implication, or at least its weaker form (ii)$\Rightarrow$(i), is true. The implication (ii)$\Rightarrow$(iii) is supported by numerous numerical experiments that we have performed. Professors Yi Zhang and Hao-Yan Chen, to whom we are sincerely grateful for noticing the latter error in the proof, expressed a similar opinion.

Recall, by Definition~\ref{def:class_R}, that the Laurent polynomial $b$ is said to belong to the class $\mathscr{R}$, i.e. $b\in\mathscr{R}$, if and only if $\Lambda(b)\subset\R$. We list several places in our article, where arguments based on the claim of Theorem~\ref{thm:gamma_impl_real_spec} have been used:
\begin{enumerate}[(a)]
\item Remark~10.
\item Sec.~4.1: Example~1. The symbol $b(z)=z^{-1}+az$, where $a\in\C\setminus\{0\}$, belongs to the class $\mathscr{R}$, if and only if $a>0$.
\item Sec.~4.2: Example~2. The symbol $b(z)=z^{-1}+\alpha z+\beta z^{2}$, where $\alpha\in\C$ and $\beta\in\C\setminus\{0\}$, belongs to the class $\mathscr{R}$, if and only if $\beta\in\R\setminus\{0\}$ and $\alpha^{3}\geq27\beta^{2}$.
\item Sec.~4.3: Example~3. The symbol $b(z)=z^{-r}(1+az)^{r+s}$, where $r,s\in\N$ and $a\in\R\setminus\{0\}$, belongs to the class $\mathscr{R}$.
\item The second paragraph of Sec.~4.4.
\end{enumerate}

The claim of Remark~10 was drawn as a direct consequence of Theorem~8 and remains open, too. The other points (b)--(e) comprise concrete examples of symbols, which belong to $\mathscr{R}$ for specific restrictions of the parameters. As the main argument for these claims, we found an explicit parametrization of the Jordan curve $\gamma$, for which $b\circ\gamma$ is real-valued in each of the cases. We will prove at least partly the claims without using Theorem~8.

The equivalence in (b) can be verified directly since the eigenvalues of $T_{n}(b)$ can be computed fully explicitly 
\[
 \lambda_{k}=2(-1)^{n}\sqrt{a}\cos\frac{\pi k}{n+1},\quad k=1,2,\dots, n,
\]
with the standard branch of the square root. This example also exhibits (iii), if $a>0$. 

The point (c) is left conjectural as inessential for the paper in its generality. Its relevant particular case, $b\in\mathscr{R}$ for $\alpha=3a^{2}$ and $\beta=a^{3}$, with $a\in\R\setminus\{0\}$, will be checked as a~special case of the point (d) in the end of this corrigendum. Lastly, a description of a~possible construction of further examples of not necessarily banded Toeplitz matrices with real spectra from the point (e) is heavily based on the falsely proven Theorem~8 and also remains open at this point.

As a final correction, we briefly indicate the proof of the implication 
\[
 b(z)=\frac{1}{z^{r}}(1+az)^{r+s}\mbox{ with } r,s\in\N \;\mbox{ and }\; a\in\R\setminus\{0\}  \quad\Rightarrow\quad b\in\mathscr{R}.
\]
without using the fact that $b$ is real on a Jordan curve. This is the claim (d). In fact, one can prove the stronger claim that the eigenvalues of $T_{n}(b)$ are all real, for all $n\in\N$, i.e.~(iii). The argument relies on the theory of osciallatory matrices~\cite{gan-kre} and has been used for the special case $r=s=1$ in the proof of \cite[Thm.~2.8]{coussementetal_tams08}. Without loss of generality, we may assume $a=1$ since two Toeplitz matrices $T_{n}(b)$ and $T_{n}(b_{a})$, where $b_{a}(z):=b(az)$, are similar, if $a\neq0$. Thus, suppose $b(z)=z^{-r}(1+z)^{r+s}$. Notice the bidiagonal matrix $T_{n}(1+z)$ is totally non-negative, see~\cite[Def.~4, p.~74]{gan-kre}. Since $T_{n}(b)$ is a submatrix of $T_{n+r+s}^{r+s}(1+z)$, it is also totally non-negative; see~\cite[p.~74]{gan-kre}. Further, without going into details, let us mention the determinant of $T_{n}(b)$ ban be explicitly computed in terms of binomial coefficients as follows:
\[
 \det T_{n}(b)=\prod_{j=0}^{r-1}\binom{n+s+j}{s}\bigg/\binom{s+j}{j}=\prod_{j=0}^{r-1}\frac{(n+s+j)!j!}{(n+j)!(s+j)!}.
\]
The determinant is obviously non-vanishing and therefore $T_{n}(b)$ is non-singular. By~\cite[Thm.~10, p.~100]{gan-kre}, $T_{n}(b)$ is oscillatory and hence eigenvalues of $T_{n}(b)$ are all positive (and simple), see~\cite[Thm.~6, p.~87]{gan-kre}.

\section*{Acknowledgments}

The authors are grateful to Professors Yi Zhang and Hao-Yan Chen for pointing out the error.

\end{document}